\documentclass[12pt]{amsart}
\usepackage{amsmath,amsfonts,amssymb,color}
\usepackage{a4wide,amsthm}
\usepackage[english]{babel}
\usepackage{textcomp}
\usepackage{hyperref} 
\usepackage{graphicx}
\usepackage{pdfpages}
\renewcommand{\div}{\operatorname{div}}

\usepackage{pgf,tikz}
\usepackage{mathrsfs}
\usetikzlibrary{arrows}
{\newtheorem{theorem}{Theorem}[section]}
{\newtheorem{notation}{Notation}[section]}
{\newtheorem{proposition}[theorem]{Proposition}}
{\newtheorem{corollary}[theorem]{Corollary}}
{\newtheorem{lemma}[theorem]{Lemma}}
{\newtheorem{definition}[theorem]{Definition}}
{\newtheorem{remark}{Remark}[section]}
\usepackage{amssymb}
\renewcommand{\div}{\operatorname{div}}
\newcommand{\esssup}{\operatorname{esssup}}
\newcommand{\sym}{\operatorname{sym}}
\newcommand{\asym}{\operatorname{ssym}}
\newcommand{\supp}{\operatorname{supp}}

\title[Sedimentation of particles in Stokes flow] 
      {Sedimentation of particles in Stokes flow}

\author[Amina Mecherbet]{}


 \email{amina.mecherbet@umontpellier.fr}

\begin{document}
\maketitle

\centerline{\scshape Amina Mecherbet}
\medskip
{\footnotesize
 \centerline{IMAG, Montpellier University}
   \centerline{Place Eug\`ene Bataillon}
   \centerline{Montpellier, 34090, France.}
} 

%

\bigskip


\begin{abstract}
In this paper, we consider $N$ identical spherical particles sedimenting in a uniform gravitational field. Particle rotation is included in the model while fluid and particle inertia are neglected. Using the method of reflections, we extend the investigation of \cite{Hofer} by discussing the threshold beyond which the minimal particle distance is conserved for a short time interval independent of $N$. We also prove that the particles interact with a singular interaction force given by the Oseen tensor and justify the mean field approximation  in the spirit of \cite{Hauray} and \cite{HJ}. 
\end{abstract}


\section{Introduction}
In this paper, we consider a system of $N$ spherical particles $(B_i)_{1\leq i \leq N}$ with identical radii $R$ immersed in a viscous fluid satisfying the following Stokes equation: 
\begin{equation}\label{eq_stokes}
\left \{
\begin{array}{rcl}
-\Delta u^N + \nabla p^N & = &0, \\ 
\div u^N & = & 0,
\end{array}
\text { on $\mathbb{R}^3 \setminus \underset{i=1}{\overset{N}{\bigcup}}\overline{B_i}$, } 
\right.
\end{equation}
completed with the no-slip boundary conditions
\begin{equation}\label{cab_stokes}
\left \{
\begin{array}{rcl}
u^N & = &  V_i +\Omega_i \times (x-x_i), \text{ on } \partial {B_i}, \\
\underset{|x| \to \infty}{\lim} |u^N(x)| & = & 0,
\end{array}\right.
\end{equation}
where $(V_i,\Omega_i) \in \mathbb{R}^3\times \mathbb{R}^3 \,,\, 1\leq i \leq N$ represent the linear and angular velocities,
\begin{eqnarray*} 
B_i:= B(x_i,R ).
\end{eqnarray*}
We describe the intertialess motion of the rigid spheres $(B_i)_{1\leq i \leq N}$ by adding to the instantaneous Stokes equation the classical Newton dynamics for the particles $(x_i)_{1\leq i\leq N}$
\begin{equation}\label{ode}
\left \{
\begin{array}{rcl}
\dot{x}_i &= & V_i,\\
F_i+ mg &=& 0,\\
T_i&=& 0,
\end{array}
\right.
\end{equation}
where $m$ denotes the mass of the identical particles adjusted for buoyancy, $g$ the gravitational acceleration, $F_i$ (resp. $T_i$) the drag force (resp. the torque) applied by the the fluid on the $i^{th}$ particle $B_i$ defined by
\begin{eqnarray*}
F_i & := & \int_{\partial B_i} \sigma(u^N,p^N) n, \\
T_i & = &  \int_{\partial B_i} (x-x_i) \times [\sigma(u^N,p^N) n],
\end{eqnarray*}
with $n$ the unit outer normal to $\partial B_i$ and
$\sigma(u^N,p^N)= 2 D(u^N) - p^N \mathbb{I}, $ the stress tensor where $2D(u^N)=\nabla u^N +{ \nabla u^N }^{\top} $.\\
Note that the constant velocities $(V_i,\Omega_i)$ of each particle are unknown and are determined by the prescribed force and torque $F_i=mg$ and $T_i=0$. In \cite{Luke}, the author shows that the linear mapping on $\mathbb{R}^{6N}$
$$(V_i,\Omega_i)_{1\leq i \leq N} \mapsto (F_i,T_i)_{1\leq i \leq N},$$ is bijective for all $N\in \mathbb{N}^*$. This ensures existence and uniqueness of $(u^N,p^N)$ and the velocities.
\begin{remark}[About the modeling and nondimensionalization]\label{rem_kappa}
Equations \eqref{eq_stokes}--\eqref{ode} describe suspensions sedimenting in a uniform gravitational field. Equations \eqref{eq_stokes}, \eqref{cab_stokes} are derived starting from the Navier-Stokes equations and neglecting the inertial terms by means of the Reynolds and Stokes number, see \cite[Chapter 1 Section 1]{Guazzelli&Morris}, \cite{Batchelor}, \cite{Luke} and all the references therein. Analogously, the ODE system \eqref{ode} is obtained by neglecting particle inertia. We refer also to \cite{DGR} where a formal derivation taking into account the slow motion of the system is performed.\\
When considering one spherical particle sedimenting in a Stokes flow, the linear relation between the drag force $F$ and the velocity $V$ is given by the Stokes law
$$
F= - 6 \pi R V\,,
$$
see Section \ref{particular_Stokes} for more details. Stokes law leads to the well-known formula for the fall speed of a sedimenting single particle under gravitational force denoted by \begin{equation}\label{kappag}
\kappa g := \frac{m}{6\pi R} g\,.
\end{equation}
It is important to point out that in our model, a scaling with respect to the velocity fall $\kappa g$ has been performed. This means that the drag forces $(F_i)_{1 \leq i \leq N}$ and the gravitational force $m g$ are terms of order $R$. Consequently, in this paper, $\kappa g$ is a constant of order one. For more details on the derivation of the model, we refer to \cite[Section 1.1]{Hofer} where a nondimensionalization including physical units is provided.
Moreover, as in \cite{DGR}, the particle radius $R$ is assumed to be proportional to $\frac{1}{N}$ so that the collective force applied by the particles on the fluid is of order one. This will be made precise in the presentation of the main assumptions.
\end{remark}
Given initial particle positions $x_i(0):= x_i^0 $, $ 1\leq i \leq N$, we are interested in the asymptotics of the solution when the number of particles $N$ tends to infinity and the radius $R$ tends to zero. The main motivation is to justifiy the representation of the motion of a dispersed phase inside a fluid using Vlasov-Stokes equations in spray theory \cite{Hamdache}, \cite{BDGM}.\\
The analysis of the dynamics is done in \cite{JO} in the dilute case \textit{i.e.} when the minimal distance between particles is at least of order $N^{-1/3}$. The authors prove that the particles do not get closer in finite time. Moreover, in the case where the minimal distance between particles is much larger than $N^{-1/3}$  the result in \cite{JO} shows that particles do not interact and sink like single particles. We refer finally to \cite{Hofer} where the author considers a particle system with minimal distance of order $N^{-1/3}$ and proves that, under a relevant time scale, the spatial density of the cloud converges in a certain averaged sense to the solution of a coupled transport-Stokes equation \eqref{mesoscopic_rho}. \\
Since the desired threshold for the minimal distance is of order $N^{-2/3}$, which allows to tackle randomly distributed particles, we are interested in extending the results for lower orders of the minimal distance. Therefore, in this paper, we continue the investigation of \cite{Hofer} by looking for a more general set of particle configurations that is conserved in time and prove the convergence to the kinetic equation \eqref{mesoscopic_rho}. Also, we include particle rotation in the modeling.
\subsection{Main assumptions and results}
In this Section, we describe the configuration of particles that we consider and present the main results : Theorem \ref{thm1} and Theorem \ref{thm2}.\\
We recall that the particles $B_i$ are spherical with identical radii $R$
$$
B_i= B(x_i,R) \,,\:\:\: 1 \leq i \leq N,
$$
where 
$$
R=\frac{r_0}{N} \,,\,\:\:\: r_0 > 0\,,
$$
with $r_0$ a positive constant satisfying a smallness assumption (see Theorem \ref{thm1}).\\ Due to the quasi-static modeling, the velocities $(V_i(t),\Omega_i(t))_{1 \leq i \leq N}$ at time $t\geq 0$ depend only on the prescribed force $(F_i)_{1 \leq i \leq N}$, torque $(T_i)_{1 \leq i \leq N}$ and the particles position $\left(x_i(t) \right)_{1 \leq i \leq N}$ at the same time $t$. Consequently, we drop the dependence with respect to time in the definition of the set of particle configurations. Keeping in mind that the idea is to start from a configuration of particles that lies in the set and show that it remains in it for a finite time interval.
\begin{definition}[Definition of the set of particle configuration]\label{def_regime}
Let $(X^N)_{N \in \mathbb{N}^*}$ be a configuration of particles, where $X^N:=(x_1, \cdots, x_N)$. We define the minimal distance $d_{\min}^N$ by
\begin{eqnarray*}
d_{\min}^N &:= &\underset{\underset{1\leq i,j\leq N}{i \neq j }}{\min } \{|x_i-x_j|\, \}\,, \forall N \in \mathbb{N}^*\,.
\end{eqnarray*}
We introduce the particle concentration $M^N$ defined for each positive sequence $(\lambda^N)_{N\in \mathbb{N}^*}$ by
$$
M^N := \underset{x \in \mathbb{R}^3}{\sup} \{ \# \{ i \in \{1,\cdots,N\} \text{ such that } x_i \in \overline{B_\infty(x,\lambda^N}) \} \}\,, \forall N \in \mathbb{N}^*\,.
$$
Given two positive constants $\bar{M}, \mathcal{E}$ and a sequence $(\lambda^N)_{N\in \mathbb{N}^*}$, we define $\mathcal{X}(\bar{M},\mathcal{E})$ as the set of configurations for which $(d_{\min}^N)_{N\in \mathbb{N}^*}$ and $(M^N)_{N \in \mathbb{N}^*}$ satisfy the following assumptions:
\begin{eqnarray}
\underset{N \in \mathbb{N}}{\sup} \, \frac{M^N}{N|\lambda^N|^3}& \leq& \bar{M},\label{bound_concentration}\\
\underset{N \in \mathbb{N}}{\sup} \, \frac{|\lambda^N|^3}{|d_{\min}^N|^2} &\leq& \mathcal{E}\,.\label{hyp1}
\end{eqnarray}
$\lambda^N$ must satisfy the following compatibility conditions:
\begin{eqnarray}\label{compatibility}
\lambda^N \geq d_{\min}^N /2 \,,\,&\underset{N \to \infty}{\lim} \lambda^N= 0\,.
\end{eqnarray}
\end{definition}
\begin{remark}
Note that, according to the definition of $M^N$, assumption \eqref{bound_concentration} ensures that
\begin{equation}\label{hyp3}
\frac{1}{N|\lambda^N|^3} \leq \bar{M},
\end{equation}
which yields thanks to assumption \eqref{hyp1}
\begin{equation}\label{hyp4}
d_{\min}^N\geq  \frac{1}{\sqrt{\mathcal{E}} \bar{M}^{1/2}}\, \frac{1}{\sqrt{N}}.
\end{equation}
Since $R \sim \frac{1}{N}$, this leads also
\begin{equation}\label{Rdmin}
\underset{N \to \infty}{\lim} \frac{R}{d_{\min}^N} =0\,,
\end{equation}
which ensures that the particles do not overlap.
\end{remark}
Furthermore, for the proof of the second Theorem \ref{thm2}, the following assumption must be satisfied initially:
\begin{equation}
\underset{N \to \infty}{\lim} \frac{|\lambda^N|^2}{d_{\min}^N(0)}= 0 \,.\label{hyp5}
\end{equation}
Finally, we define $\rho^N$ the spatial density of the cloud by
\begin{eqnarray*}
\rho^N(t,x) = \frac{1}{N} \underset{i=1}{\overset{N}{\sum}} \delta_{x_i(t)}(x)\,,\,&\rho_0^N:= \rho^N(0,x).
\end{eqnarray*}
In the rest of this paper, if needed, we make clear the dependence with respect to time by writing for all $ N\in \mathbb{N}^*$, $X^N(t)=(x_1(t),\cdots, x_N(t))$ for the particles configuration, $d_{\min}^N(t)$ for the minimal distance and $M^N(t)$ the particle concentration at time $t\geq0 $.\\
The main results of this paper are the two following theorems. The first one ensures that the particle configurations considered herein are preserved in a short time interval depending only on the data $r_0$, $\bar{M}$, $\mathcal{E}$, $\kappa |g|$.
\begin{theorem}\label{thm1}
Let $(X^N(0))_{N \in \mathbb{N}^*}$ be the initial position of the particles. Assume that there exists $\bar{M}$, $\mathcal{E}$ and a sequence $(\lambda^N)_{N \in \mathbb{N}^*}$ such that $(X^N(0))_{N \in \mathbb{N}^*}$ lies in the set $\mathcal{X}(\bar{M}, \mathcal{E})$ \textit{i.e.} assumptions \eqref{bound_concentration}, \eqref{hyp1}, \eqref{compatibility} hold true initially.\\
If $\bar{M}^{1/3} r_0$ is small enough, there exists $N^* \in \mathbb{N}^*$ depending on $(r_0,\bar{M},\mathcal{E})$ and $T>0$ depending on $(r_0, \mathcal{E}, \bar{M}, \kappa |g|)$ such that for all $t\in[0,T]$ and $N\geq N^*$
$$
d_{\min}^N(t) \geq \frac{1}{2} d_{\min}^N(0),
$$
$$
{M^N(t)} \leq 8^4 {M^N(0)}.
$$
\end{theorem}
The second part of the result is the justification of the convergence of $\rho^N$ when $N$ tends to infinity.
\begin{theorem}\label{thm2}
Consider the maximal time $T>0$ introduced in Theorem \ref{thm1} and the additional assumption \eqref{hyp5}. Let $\rho_0$ be a positive regular density such that $\int_{\mathbb{R}^3}\rho_0=1$. We denote by $(\rho,u)$ the unique solution to the coupled equation \eqref{mesoscopic_rho}.\\
There exists some positive constants $C_1,C_2$ depending on $( r_0,\bar{M},\mathcal{E}, \|\rho_0\|_{L^\infty}, \kappa |g|)$ and $N^* \in \mathbb{N}^*$ depending on $( r_0, \bar{M}, \mathcal{E}, \|\rho_0\|_{L^\infty}, \kappa |g|, T)$  such that for all $N\geq N^*$ and $t\in[0,T]$
$$
W_1(\rho^N(t,\cdot),\rho(t,\cdot)) \leq C_1 \left (\lambda^N +  d_{\min}^N(0) \,t+ W_1(\rho_0,\rho_0^N)  \right)e^{C_2t}.$$
\end{theorem}
This shows that if the initial particle distribution $\rho_0^N$ converges to $\rho_0$ then the particle distribution $\rho^N$ converges toward the unique solution $\rho$ of equation \eqref{mesoscopic_rho} for all time $0\leq t \leq T$. Moreover, Theorem \ref{thm2} provides a quantitative convergence rate in terms of the initial Wasserstein distance $W_1(\rho_0,\rho_0^N)$.
\begin{remark}
The regularity assumption on the initial density $\rho_0$ is the one introduced by H\"ofer in \cite{Hofer} which is $ \rho_0$, $\nabla \rho_0 \in X_\beta$, for some $\beta>2 $. See Section \ref{Vlasov_existence} for the definition of $X_\beta$. In particular, the assumption is satisfied if $\rho_0$ is compactly supported and $\mathcal{C}^1$.
\end{remark}
The idea of proof of Theorem \ref{thm2} is to formulate the problem considered as a mean-field problem. The mean-field theory consists in approaching equations of motion of large particles systems $(X_1,\cdots,X_N)$ when the number of particles $N$ tends to infinity. In mean-field theory, the ODE governing the particle motion is known and is given by
\begin{equation}\label{system1}
\left\{
\begin{array}{rcl}
\dot{X}_i & =& \frac{1}{N} \underset{i=1}{\overset{N}{\sum}} F(X_i-X_j), \\
X_i(0) &=& X_i^0,
\end{array}
\right.
\end{equation}
where the kernel $F$ is the interaction force of the particles. The limit model describing the time evolution for the spatial density $\rho(t,x)$ is given by
\begin{equation}\label{mean_field}
\left\{
\begin{array}{c}
\partial_t \rho+\mathcal{K}\rho \cdot \nabla \rho = 0\,, \, \\\\
\mathcal{K}\rho(x):= \int_{\mathbb{R}^3} F(x-y) \rho(t,y) dy, \
\end{array}
\right.
\end{equation}
In our case, the first difficulty is to extract a system similar to \eqref{system1} for the particle motion and to identify the interaction force $F$. A key step is then a sharp expansion of the velocities for large $N$. We obtain for each $ 1 \leq i \leq N$
\begin{equation}\label{vitesse_V_i}
V_i = \kappa g+6 \pi \frac{r_0}{N} \, \underset{j\neq i}{\sum} \Phi(x_i-x_j)\kappa g  + O\left (d_{\min}^N \right) \,,\:\:\,\:\: 1 \leq i \leq N, 
\end{equation}
where $\Phi$ is the Green's function for the Stokes equations, also called the Oseen tensor (see formula \eqref{Oseen} for a definition). $\kappa g$ is the fall speed of a sedimenting single particle under gravitational force and is of order one in our model, see Remark \ref{rem_kappa}. This shows that the particle system satisfies \emph{approximately} equation \eqref{system1} with an interaction force given by the Oseen tensor. Since the convolution term $\mathcal{K} \rho$ appearing in \eqref{mean_field} corresponds to the solution of a Stokes equation in our case, the limiting model describing \eqref{eq_stokes}, \eqref{cab_stokes}, \eqref{ode} is a coupled transport-Stokes equation
\begin{equation}\label{mesoscopic_rho}
\left\{
\begin{array}{rcl}
\frac{\partial \rho}{\partial t }+{\div}( ( \kappa g + u) \rho ) &=& 0\,, \\
-\Delta u +\nabla p &=& 6\pi r_0 \kappa \rho g\,,\\
\div(u) &=& 0\,,\\
\rho(0,\cdot) &= & \rho_0 \,,
\end{array}
\right.
\end{equation}
The proof of Theorem \ref{thm2} is based on the two papers \cite{HJ}, \cite{Hauray} where, in the first one, the authors justify the mean field approximation and prove the propagation of chaos for a system of particles interacting with a singular interaction force and where the ODE governing the particle motion is second order. In \cite{Hauray} the author considers a different mean-field equation where the particle dynamics is a first order ODE. The results obtained hold true for a family of singular kernels and applies to the case of vortex system converging towards equations similar to the 2D Euler equation in vorticity formulation. The associated kernel in this case is the Biot-Savard kernel.\\
In order to extract the first order terms for the velocities $(V_i,\Omega_i)$ we apply the method of reflections. This method is introduced by Smoluchowski \cite{Smo} in 1911. The main idea is to express the solution $u^N$ of $N$ separated particles as superposition of fields produced by the isolated $N$ particle solutions. We refer to \cite[Chapter 8]{KK} and \cite[Section 4]{Guazzelli&Morris} for an introduction to the method. A convergence proof based on orthogonal projection operators is introduced by Luke \cite{Luke} in 1989. We refer also to the method of reflections developped in \cite{HV} which is used by H{\"o}fer in \cite{Hofer}.\\
In this paper, we design a modified method of reflections that takes into account the particle rotation and relies on explicit solutions of Stokes flow generated by a translating, rotating and straining sphere. To obtain the convergence of the method of reflections we need to identify particle configuration that can be propagated in time. The particle configuration considered herein is the one introduced in \cite{Hillairet} to study the homogenization of the Stokes problem in perforated domains. The novelty is that the author considers the minimal distance $d_{\min}^N$ together with the particle concentration $M^N$ as parameters to describe the cloud. The result in \cite{Hillairet} extends in particular the validity of the homogenization problem for randomly distributed particles \textit{i.e.} particle configurations having a minimal distance of order at least $N^{-2/3}$. Note that the notion of particle concentration appears also in \cite{HJ} to describe the cloud.
\subsection{Discussion about the particle configuration set}
As stated above, the assumptions introduced in Definition \ref{def_regime} are based on \cite{Hillairet}. Assumptions \eqref{bound_concentration} and \eqref{compatibility} means that there exists a uniformly bounded discrete spatial density that approximates $\rho^N$. Indeed, if we define $\tilde{\rho}^N$ by
\begin{equation}
\tilde{\rho}^N(t,x):= \frac{1}{N} \underset{i=1}{\overset{N}{\sum}} \frac{1_{B(x_i,\lambda^N)}}{\left|B(x_i,\lambda^N) \right|}\,,
\end{equation} 
one can show that 
$$
W_1(\tilde{\rho}^N, \rho^N) \leq  \lambda^N.
$$
Assumption \eqref{bound_concentration} ensures that there exists a sequence $\lambda^N$ for which the infinite norm of $\tilde{\rho}^N$ is bounded by $\bar{M}$, see formula \eqref{bound_tildethoN}. This suggests that $\|\rho\|_\infty$ and $\bar{M}$ are equivalent. \\
We recover the result of \cite{JO} in the case where $\lambda^N=N^{-1/3}$ and the minimal distance $d_{\min}^N$ is much larger than $N^{-1/3}$, the explicit formula for the velocities \eqref{vitesse_V_i} implies in this case 
$$
|V_i-\kappa g | \lesssim  \frac{6\pi r_0}{N} \underset{j\neq i}{\sum} \frac{1}{|x_i-x_j|} |\kappa g| \lesssim \frac{1}{N} \frac{N^{2/3}}{d_{\min}^N} \ll 1\,,
$$
which is in accordance with the ``non-interacting scenario'' explained in \cite{JO}. In our case, the smallness assumption on $r_0 \bar{M}^{1/3}$ means that we consider a density of particles such that $\|\rho\|_{\infty}$ is small but of order one. Indeed, the second term in the velocity formula \eqref{vitesse_V_i} can be seen as a perturbation of order one of the velocity fall $\kappa g $ in the case where $\bar{M}$ (or the particle density $\|\rho\|_\infty$) is small. This can be also seen in the coupled equation \eqref{mesoscopic_rho} where the velocity term $u$ is proportional to $\|\rho\|_\infty$.\\
The second assumption \eqref{hyp1} ensures the conservation of the minimal distance, see Proposition \ref{prop1}. In particular, for $\lambda^N=N^{-1/3}$, Theorem \ref{thm1} extends the previous known results to configurations having minimal distance at least of order $N^{-1/2}$, see assumption \eqref{hyp1}.
This lower bound for the minimal distance appears naturally in our analysis and is closely related to the properties of the Green's function for the Stokes equations. 
We emphasize that this critical minimal distance appears also in the mean-field analysis due to \cite{Hauray}.
Precisely, computations in the proof of \cite[Theorem 2.1]{Hauray} show the convergence for a short time interval under the assumption that $$\frac{W_\infty(\rho_0,\rho_0^N)^3}{|d_{\min}^N(0)|^2}, $$ is uniformly bounded, see Definition \ref{def_Wasserstein} for the definition of the infinite Wasserstein distance $W_\infty$. Standard measure-theory arguments show that the infinite Wasserstein distance ensures assumption \eqref{bound_concentration}. In other words, one can take $\lambda^N$ to be the infinite Wasserstein distance, which yields finally the same assumption \eqref{hyp1}. \\
The first assumption in formula \eqref{compatibility} means that we are interested in cases where there is more than one particle per cube of length $\lambda^N$. As pointed out by Hillairet in \cite{Hillairet}, one can choose a larger sequence $(\lambda^N)_{N\in\mathbb{N}^*}$ such that the compatibility assumption holds true. Note also that, in the case where $\lambda^N$ is the infinite Wasserstein distance, this compatibility assumption is satisfied by definition.\\ Finally, assumption \eqref{hyp5} is needed for the control of the Wasserstein distance.
\subsection{Outline of the paper and main notations}  
The remaining Sections of this paper are organized as follows.\\
In Section 2 we recall the classical results for the existence and uniqueness of the Stokes solution $u^N$. We recall also the definition of the drag force $F_i$, torque $T_i$ and stresslet $S_i$ and present in Section 2.1 the particular solutions to a Stokes flow generated by a translating, a rotating or a straining sphere. Finally, the end of Section 2 is devoted to the approximation of the stresslets $S_i$. In Section 3 we present and prove the convergence of the method of reflections in order to compute the first order terms for the velocities $(V_i,\Omega_i)_{1 \leq i \leq N}$. Section 4 is devoted to the proof of Theorem \ref{thm1}. In Section 5 we recall some definitions associated to the Wasserstein distance. We present then the strong existence, uniqueness and stability theory for equation \eqref{mean_field}. In the second part of Section 5 we show that the discrete density $\rho^N$ satisfies weakly a transport equation. Section 6 is devoted to the proof of the second Theorem \ref{thm2}. Finally, some technical Lemmas are presented in the appendix.
\begin{notation}
In this paper, $n$ always refer to the unit outer normal to a surface.\\
The following shortcut will be often used 
$$
d_{ij}= |x_i-x_j| \,, 1 \leq i \neq j \leq N \,,
$$
where we drop the dependence with respect to $N$ in order to simplify the notation.\\ 
Given an exterior domain $\Omega$ with smooth boundaries, we set
 $$
 \mathcal{C}^\infty(\overline{\Omega}):= \{v_{| \Omega}\,,\, v \in \mathcal{C}^\infty_c(\mathbb{R}^3) \}
, $$
 and the following norm for all $u\in  \mathcal{C}^\infty(\overline{\Omega})$ 
 $$
 \|u\|_{1,2} := \|\nabla u \|_{L^2(\Omega)},
 $$
we define then the homogeneous Sobolev space $D(\Omega)$ as the closure of $ \mathcal{C}^\infty(\overline{\Omega})$ for the norm $\| \cdot\|_{1,2}$ (see \cite[Theorem II.7.2]{Galdi}).
 We also use the notation $D_\sigma(\Omega)$ for the subset of divergence-free $D(\Omega)$ fields  
$$
D_\sigma(\Omega):= \{ u \in D(\Omega)\,,\, \div u=0 \}.
$$
Which is also the closure of the subset of divergence-free $ \mathcal{C}^\infty(\overline{\Omega})$ fields for the $\| \cdot \|_{1,2}$ norm.
Analogously, if $\Omega=\mathbb{R}^3$ we use the notation $$\dot{H}^1_\sigma(\mathbb{R}^3)= D_\sigma(\mathbb{R}^3).$$
For all $3 \times3$ matrix $M,$ we define $\sym(M)$ (resp. $\asym(M)$) as the symmetric part of $M$ (resp. the skew-symmetric part of $M$)  
\begin{eqnarray*}
\sym(M)= \frac{1}{2}( M + M^{ \top})\,,\,& \asym(M)= \frac{1}{2}( M - M^{\top}) .
\end{eqnarray*}
We denote by $\times$ the cross product on $\mathbb{R}^3$ and by $\otimes$ the tensor product on $\mathbb{R}^3$ which associates to each couple $(u,v) \in \mathbb{R}^3 \times \mathbb{R}^3$ the $3\times3 $ matrix defined as  
$$
(u\otimes v)_{ij} = u_i v_j \,,\: 1\leq i,j\leq 3.
$$
For all $3 \times 3$ matrices $A,B$, we use the classical notation
$$
A:B = \underset{i=1}{\overset{3}{\sum}}\underset{j=1}{\overset{3}{\sum}} A_{ij} B_{ij}.
$$
In $\mathbb{R}^3$, $|\cdot|$ stands for the Euclidean norm while $|\cdot|_\infty$ represents the $l^\infty$ norm. We use the notation $B_\infty(x,r)$ for the ball with center $x$ and radius $r$ for the $l^\infty$ norm.\\
Finally, in the whole paper we use the symbol $\lesssim$ to express an inequality with a multiplicative constant independent of $N$ and depending only on $r_0$, $\bar{M}$, $\mathcal{E}$ and eventually on $\kappa |g|$ which is uniformly bounded according to Remark \ref{rem_kappa}.
\end{notation}
 \section{Reminder on the Stokes problem}
In this Section we recall some results concerning the Stokes equations.
We remind that for all $N \in \mathbb{N}$  we denote by $(u^N,p^N)$ the solution to \eqref{eq_stokes} -- \eqref{cab_stokes}. Keeping in mind that the linear mapping, that associates to the linear and angular velocities the forces and torques, is bijective (see \cite{Luke}) together with the classical theory for the Stokes equations yields:
\begin{proposition}
For all $N \in \mathbb{N}$, there exists a unique pair $(u^N,p^N) \in {D}_\sigma(\mathbb{R}^3 \setminus\underset{i}{\bigcup}  \overline{B_i})\times L^2(\mathbb{R}^3 \setminus  \underset{i}{\bigcup} \overline{B_i} )$ and unique velocities $(V_i,\Omega_i)_{1 \leq i \leq N}$ such that
\begin{eqnarray*}
\int_{\partial B_i} \sigma(u^N,p^N) n + mg &=&0\,, \forall \, 1\leq i \leq N\,,\\
\int_{\partial B_i} (x-x_i) \times [\sigma(u^N,p^N) n] &=&0 \,,  \forall \,1 \leq i\leq N \,,
\end{eqnarray*}
and $u$ realizes
\begin{multline}\label{variational_form}
\inf \Bigg \{ \int_{\mathbb{R}^3 \setminus  \underset{i}{\bigcup} \overline{B_i} }|\nabla v |^2,\\ v \in {D}_\sigma(\mathbb{R}^3 \setminus \underset{i}{\bigcup}\overline{B_i}) \,,\, v= V_i+\Omega_i\times (x-x_i) \text{ on } \partial B_i \,,\, 1 \leq i \leq N  \Bigg \}.
\end{multline}
\end{proposition}
The velocity field $u^N$ can be extended to $V_i+ \Omega_i \times (x-x_i)$ on each particle $B_i$. This extension denoted also $u^N$ is in $\dot{H}^1_\sigma(\mathbb{R}^3)$. \\
We recall the definition of the force $F_i\in \mathbb{R}^3$, torque $T_i \in \mathbb{R}^3$ and stresslet $S_i \in \mathcal{M}_3(\mathbb{R})$ applied by the particle $B_i$ on the fluid (see \cite[Section 1.3]{Guazzelli&Morris})
\begin{align}\label{FM}
F_i & =  \int_{\partial B_i} \sigma (u^N,p^N) n. \nonumber \\
M_i & =   \int_{\partial B_i} (x-x_i) \otimes \left [ \sigma (u^N,p^N) n \right ].
\end{align}
The matrix $M_i$ represents the first momentum which is decomposed into a symmetric and skew-symmetric part $$M_i= T_i + S_i,$$
the symmetric part $S_i$ is called stresslet, see \cite[Section 2.2.3]{Guazzelli&Morris}. Since the skew-symmetric part of a $3 \times 3$ matrix M has only three independent components, it can be associated to a unique vector T such that 
$$
\asym(M)\, x = T \times x \,,\, \forall \, x \in \mathbb{R}^3 .
$$
In this paper, we allow the confusion between the skew-symmetric matrix $\asym(M)$ and the vector $T$. Hence, we define the torque $T_i \in \mathbb{R}^3$ as being the skew-symmetric part of the first momentum $M_i$ which satisfies 
\begin{align}\label{TS}
T_i & = \asym(M_i)=   \int_{\partial B_i} (x-x_i) \times \left [ \sigma (u^N,p^N) n \right ], \nonumber \\
S_i&=\sym(M_i).
\end{align}
\subsection{Particular Stokes solutions}\label{particular_Stokes}
The linearity of the Stokes problem allows us to develop powerful tools that will be used in the method of reflections. In particular, we investigate in what follows the analytical solution to a Stokes flow generated by a translating, a rotating or a straining sphere.
The motivation in considering these cases is that the fluid motion near a point $x_0$ may be approximated by
$$
u(x)\sim u(x_0) + \nabla u (x_0) \cdot(x-x_0),
$$
hence, if we replace the boundary condition on each particle by its Taylor series of order one, we can use these special solutions to approximate the flow $u$. The results and formulas of this Section are detailed in  \cite[Section 2]{Guazzelli&Morris} and \cite[Section 2.4.1]{KK}.
In what follows $B:=B(a,r)$ is a ball centered in $a \in \mathbb{R}^3$ with radius $r>0$.
\subsubsection{Case of translation} 
Let $V \in \mathbb{R}^3$. We consider $(U_{a,R}[V], P_{a,R}[V])$ the unique solution to the following Stokes problem: 
 \begin{equation}
\left \{
\begin{array}{rcl}
-\Delta U_{a,R}[V] + \nabla P_{a,R}[V]& = &0, \\ 
\div U_{a,R}[V] & = & 0,
\end{array}
\text { on $\mathbb{R}^3 \setminus \overline{B}$, } 
\right.
\end{equation}
completed by the boundary condition
\begin{equation}
\left \{
\begin{array}{rcl}
U_{a,R}[V] & = &  V, \text{ on $ \partial B$,} \\
\underset{|x| \to \infty}{\lim} |U_{a,R}[V](x)| & = & 0.
\end{array}
\right.
\end{equation}
$U_{a,R}[V]$ is the flow generated by a unique sphere immersed in a fluid moving at $V$. The explicit formula for $(U_{a,R}[V], P_{a,R}[V])$ is derived in \cite[Section 3.3.1]{KK} and also in \cite[Formula (2.12) and (2.13)]{Guazzelli&Morris}. Explicit formulas imply the existence of a constant $C>0$ such that for all $ x\in \mathbb{R}^3 \setminus {B(a,R)}$
\begin{eqnarray}\label{Stokeslet_maj}
|U_{a,R}[V](x)| \leq C  R \frac{|V|}{|x-a|},\,&
|\nabla U_{a,R}[V](x)|+ |P_{a,R}[V](x)| \leq C  R \frac{|V|}{|x-a|^2}.
\end{eqnarray}
\begin{equation}\label{Stokeslet_maj_2}
|\nabla^2 U_{a,R}[V](x)|\leq C R \frac{|V|}{|x-a|^3}.
\end{equation}
On the other hand, the force $F$, torque $T$ and stresslet $S$ exerted by a translating sphere $B$ as defined in \eqref{FM} read
\begin{eqnarray}\label{stokeslet_force}
F= -6 \pi R V \:,\: T =0 \:,\: S=0.
\end{eqnarray}
We recall now an important formula that links the solution to the Green's function of the Stokes problem. For all $x\in \mathbb{R}^3 \setminus B(a,R)$ we have
\begin{equation}\label{Oseen_Stokeslet}
U_{a,R}[V](x)= - \left(\Phi(x-a) - \frac{R^2}{6} \Delta \Phi(x-a)\right) F\,,
\end{equation}
where $\Phi$ is the Green's function for Stokes flow also called Oseen-tensor
\begin{equation}\label{Oseen}
\Phi(x) = \frac{1}{8\pi} \left (\frac{1}{|x|} \mathbb{I}_3 + \frac{1}{|x|^3} x\otimes x \right ).
\end{equation}
The $3\times 3 $ matrix $\Delta \Phi$ represents the Laplacian of $\Phi$ and is given by
$$
\Delta \Phi(x)= \frac{1}{8\pi} \left (\frac{2}{|x|^3} \mathbb{I}_3 - \frac{6}{|x|^5} x\otimes x \right).
$$
The first term in the right-hand side of \eqref{Oseen_Stokeslet} is the point force solution also called stokeslet, see \cite[Section 3.1]{Guazzelli&Morris}. In this paper we use the term stokeslet to define the whole solution $U_{a,r}[V]$ which can bee seen as an extension of the point force solution.
\begin{remark}
Formula \eqref{Oseen_Stokeslet} is closely related to the Fax\'en law which represents the relations between the force $F$ and the velocity $V$. We refer to \cite[Section 2.3]{Guazzelli&Morris} and \cite[Section 3.5]{KK} for more details on the topic.\\
Remark also that in \eqref{Oseen_Stokeslet} the point force solution retains the most slowly decaying portion, which is of order $\frac{R}{|x|}$. This property is useful in order to extract the first order terms for the velocities $(V_i)_{1 \leq i \leq N}$, see Lemma \ref{Oseen_Stokeslet_maj}. 
\end{remark}
Moreover, we recall a Lipschitz-like inequality satisfied by the Oseen tensor
\begin{equation}\label{lipschitz}
|\Phi(x)- \Phi(y)| \leq C \frac{|x-y|}{\min(|y|^2\,,\, |x|^2)},  \:\:\forall\, x \,,\, y \neq 0.
\end{equation}
Finally, in this paper, the velocity field $U_{a,R}[V]$ is extended by $V$ on $B(a,R)$. 
\subsubsection{Case of rotation}
Let $\omega \in \mathbb{R}^3$. Denote by $(A^{(1)}_{a,R}[\omega], P^{(1)}_{a,R}[\omega] ) $ the unique solution to
\begin{equation}
\left \{
\begin{array}{rcl}
-\Delta A^{(1)}_{a,R}[\omega] + \nabla P^{(1)}_{a,R}[\omega] & = &0, \\ 
\div A^{(1)}_{a,R}[\omega] & = & 0,
\end{array}
\text { on $\mathbb{R}^3 \setminus \overline{B(a,R)},$ } 
\right.
\end{equation}
completed with the boundary conditions
\begin{equation}
\left \{
\begin{array}{rcl}
A^{(1)}_{a,R}[\omega] & = &  \omega \times (x-a), \text{ on $ \partial B(a,R),$} \\
\underset{|x| \to \infty}{\lim} |A^{(1)}_{a,R}[\omega]| & = & 0 .
\end{array}
\right.
\end{equation}
 $A^{(1)}_{a,R}[\omega]$ represents the flow generated by a sphere rotating with angular velocity  $\omega$. In particular we have $P^{(1)}_{a,R}[\omega]=0$ due to symmetries. The drag force $F$ and stresslet $S$ also vanish 
\begin{eqnarray*}
F=0 \,,\,& S=0.
\end{eqnarray*}
On the other hand, the hydrodynamic torque resulting
from the fluid traction on the surface defined in \eqref{TS} is given by 
\begin{equation}\label{rotelet_torque}
T =- 8 \pi R^3 \omega.
\end{equation}
Finally, there exists $C>0$ such that for all $x \in \mathbb{R}^3 \setminus {B(a,R)}$
\begin{eqnarray*}
|A^{(1)}_{a,R}[\omega]| \leq C R^3 \frac{|\omega|}{|x-a|^2}
,\,& 
|\nabla A^{(1)}_{a,R}[\omega]|+ | P^{(1)}_{a,R}[\omega] | \leq C R^3 \frac{|\omega|}{|x-a|^3}.
\end{eqnarray*}
\subsubsection{Case of strain} 
Let $E$ be a trace-free $3\times 3$ symmetric matrix.\\
Denote by $(A^{(2)}_{a,R}[E], P^{(2)}_{a,R}[E])$ the unique solution to
\begin{equation}
\left \{
\begin{array}{rcl}
-\Delta A^{(2)}_{a,R}[E] + \nabla P^{(2)}_{a,R}[E] & = &0, \\ 
\div A^{(2)}_{a,R}[E] & = & 0,
\end{array}
\text { on $\mathbb{R}^3 \setminus \overline{B(a,R)},$ } 
\right.
\end{equation}
completed with the boundary conditions
\begin{equation}
\left \{
\begin{array}{rcl}
A^{(2)}_{a,R}[E] & = &  E (x-a), \text{ on $ \partial B(a,R),$} \\
\underset{|x| \to \infty}{\lim} |A^{(2)}_{a,R}[E]| & = & 0.
\end{array}
\right.
\end{equation}
The velocity field $A^{(2)}_{a,R}[E]$ is the flow generated by a sphere submitted to the strain $E (x-a)$. In this case, the drag force and torque vanishes
\begin{eqnarray}\label{stresslet_force_torque}
F=0 \,,\,& T=0.
\end{eqnarray}
On the other hand, the symmetric part of the first momentum $S$ as defined in \eqref{TS} is given by 
\begin{equation}\label{stresslet_strain}
S = - \frac{20}{3} \pi R^3 E.
\end{equation}
Finally, there exists $C>0$ such that for all $x \in \mathbb{R}^3 \setminus B(a,R) $ we have
\begin{eqnarray}
|A^{(2)}_{a,R}[E]| \leq C R^3 \frac{|E|}{|x|^2}\,,&
|\nabla A^{(2)}_{a,R}[E]| + |P^{(2)}_{a,R}[E] (x) |\leq C R^3 \frac{|E|}{|x|^3}.
\end{eqnarray}
\subsubsection{Final notations} 
Now, assume that $D$ is a trace-free $3 \times3$ matrix. We denote by $(A_{a,R}[D], P_{a,R}[D])$ the unique solution to 
\begin{equation}
\left \{
\begin{array}{rcl}
-\Delta A_{a,R}[D] + \nabla P_{a,R}[D] & = &0, \\ 
\div A_{a,R}[D] & = & 0,
\end{array}
\text { on $\mathbb{R}^3 \setminus \overline{B(a,R)}$, } 
\right.
\end{equation}
completed by the boundary conditions
\begin{equation} \label{superposition}
\left \{
\begin{array}{rcl}
A_{a,R}[D] & = &  D (x-a), \text{ on $ \partial B(a,R),$} \\
\underset{|x| \to \infty}{\lim} |A_{a,R}[D]| & = & 0 .
\end{array}
\right.
\end{equation}
We set then $D = E+\omega$ with $E= \sym(D)$ and $\omega=\asym(D)$. As stated in the definition \eqref{TS}, $\omega$ represents also a $3D$ vector.
Hence, the boundary condition \eqref{superposition} reads
$$
A_{a,R}[D](x)  =   D (x-a) = E (x-a) + \omega \times (x-a),  \:\: \text{ for all $x\in \partial B(a,R)$.}
$$
We have, thanks to the linearity of the Stokes equation, that 
 $$
 (A_{a,R}[D], P_{a,R}[D]) = (A^{(1)}_{a,R}[\omega], P^{(1)}_{a,R}[\omega])+ (A^{(2)}_{a,R}[E], P^{(2)}_{a,R}[E]).
 $$
Since the two solutions have the same decay-rate, this yields for all $x \in \mathbb{R}^3 \setminus {B(a,R)}$
\begin{eqnarray}\label{decay_rate}
|A_{a,R}[D]| \leq C R^3 \frac{|D|}{|x|^2},\,&
|\nabla A_{a,R}[D]| + |P_{a,R}[D] (x) |\leq C R^3 \frac{|D|}{|x|^3}.
\end{eqnarray}
Analogously, for the second derivative we have
\begin{equation}\label{decay_rate_2}
 |\nabla ^2 A_{a,R}[D](x)|\leq C  R^3 \frac{|D|}{|x-a|^4}.
\end{equation}
\subsection{Approximation result}
In this part we consider the unique solution $(v,p)$ of the following Stokes problem: 
\begin{equation}
\left \{
\begin{array}{rcl}
-\Delta v + \nabla p & = &0, \\ 
\div v & = & 0,
\end{array}
\text { on $\mathbb{R}^3 \setminus {\underset{i=1}{\overset{N}{\bigcup}} \overline{B_i}}$, } 
\right.
\end{equation}
completed with the boundary conditions
\begin{equation}
\left \{
\begin{array}{rcl}
v & = &  V+D(x-x_1), \text{ on $ \partial B_1,$} \\
v & = & 0, \text{ on $\partial B_i$, $ i\neq 1,$} \\
\underset{|x| \to \infty}{\lim} |v(x)| & = & 0,
\end{array}
\right.
\end{equation}
with $V \in \mathbb{R}^3$ and $D$ a trace-free $3\times 3 $ matrix. We set $$v_1:= U_{x_1,R}[V] + A_{x_1,R}[D].$$ 
We aim to show that the velocity field $v_1$ is a good approximation of the unique solution $v$. 
\begin{lemma}\label{approximation_result}
For $N$ sufficiently large, we have the following error bound:
$$
\| \nabla v - \nabla v_1 \|_{L^2(\mathbb{R}^3 \setminus \bigcup_i \overline{B_i})} \lesssim \frac{R}{\sqrt{d_{\min}^N}} |V| + \frac{R^3}{|d_{\min}^N|^{3/2}}|D|.
$$\end{lemma}
\begin{proof}
We have 
\begin{multline*}
\| \nabla v - \nabla v_1 \|^2_{L^2(\mathbb{R}^3 \setminus \bigcup_i \overline{B_i})} = \| \nabla v \|^2_{L^2(\mathbb{R}^3 \setminus \bigcup_i \overline{B_i})}\\ - 2 \int_{\mathbb{R}^3 \setminus \bigcup_i \overline{B_i}} \nabla v : \nabla v_1 + \|\nabla v_1 \|^2_{L^2(\mathbb{R}^3 \setminus \bigcup_i \overline{B_i})},
\end{multline*}
as $v$ and $v_1$ satisfy the same boundary condition on $\partial B_1$ this yields 
\begin{multline}
\int_{\mathbb{R}^3 \setminus \bigcup_i \overline{B_i}} \nabla v : \nabla v_1 =- \int_{\partial B_1}( \partial_n v_1 - p_1 n) \cdot v\\ =-\int_{\partial B_1}( \partial_n v_1 - p_1 n) \cdot v_1 =\|\nabla v_1 \|^2_{L^2(\mathbb{R}^3 \setminus \bigcup_i \overline{B_i})} ,
\end{multline}
hence 
$$
\| \nabla v - \nabla v_1 \|^2_{L^2(\mathbb{R}^3 \setminus \bigcup_i \overline{B_i})} = \| \nabla v \|^2_{L^2(\mathbb{R}^3 \setminus \bigcup_i \overline{B_i})} - \| \nabla v_1 \|^2_{L^2(\mathbb{R}^3 \setminus \bigcup_i \overline{B_i})}.
$$
In order to bound the first term we construct an extension $\tilde{v}$ of the boundary conditions of $v$ and apply the variational principle. We define
$$
\tilde{v}:= \chi\left (\frac{\cdot - x_1}{d_{\min}^N/4} \right ) v_1 - \mathcal{B}_{x_1, d_{\min}^N/4,d_{\min}^N/2}[\bar{f} ],
$$
where $\chi$ is a truncation function such that $\chi = 1 $ on $B(0,1)$ and $\chi=0$ out of $B(0,2)$. Thanks to formula \eqref{Rdmin}, for $N$ sufficiently large we have $R < d_{\min}^N/4$ and thus $\supp \tilde{v} \subset B(x_1, d_{\min}^N/2)$. $\bar{f}$ is defined as follows 
$$
\bar{f}(x):= v_1(x) \cdot \nabla \left [x\mapsto  \chi\left (\frac{x - x_1}{d_{\min}^N/4} \right )  \right],
$$
and $\mathcal{B}_{x_1, d_{\min}^N/4,d_{\min}^N/2}$ denotes the Bogovskii operator satisfying  
$$
\div \mathcal{B}_{x_1, d_{\min}^N/4,d_{\min}^N/2}[f] = f,
$$ 
for all $f\in L^q_0(B(x_1,d_{\min}^N / 2 ) \setminus \overline{B(x_1,d_{\min}^N/4})$ , $q\in(0,\infty)$. We refer to \cite[Theorem III.3.1]{Galdi} for a complete definition of the Bogovskii operator. In particular, from \cite[Lemma 16]{Hillairet}, there exists a positive constant $C$ independent of $d_{\min}^N$ such that
\begin{equation}\label{bogovski}
\| \nabla  \mathcal{B}_{x_1, d_{\min}^N/4,d_{\min}^N/2}[\bar{f}]  \|_{L^2(A_1)} \leq C \|\bar{f} \|_{L^2(A_1)},
\end{equation}
where $A_1:= B(x_1,{d_{\min}^N}/{2}) \setminus \overline{B(x_1,{d_{\min}^N}/{4})}$.
With this construction $\tilde{v}$ is a divergence-free field satisfying the same boundary conditions as $v$. Moreover, applying formula \eqref{bogovski}, there exists (another) constant $C>0$ such that
\begin{equation*}
\begin{split}
&\|\nabla \tilde{v} \|^2_{L^2(\mathbb{R}^3 \setminus\bigcup_i \overline{B_i})}\\
& =
\int_{\mathbb{R}^3\setminus  \bigcup_i \overline{B_i}} \left|\nabla \left [x \mapsto \chi\left (\frac{x - x_1}{d_{\min}^N/4} \right ) v_1(x) \right ]\right|^2 dx\\
& + \int_{\mathbb{R}^3\setminus \bigcup_i \overline{B_i}}| \nabla \mathcal{B}_{x_1, d_{\min}^N/4,d_{\min}^N/2}[\bar{f} ](x)|^2 dx\\
& - 2 \int_{\mathbb{R}^3\setminus \bigcup_i \overline{B_i}} \nabla \left [x \mapsto \chi\left (\frac{x - x_1}{d_{\min}^N/4} \right ) v_1(x) \right ] : \nabla \mathcal{B}_{x_1, d_{\min}^N/4,d_{\min}^N/2}[\bar{f} ](x) dx\,,\\
&\leq \int_{\mathbb{R}^3\setminus B_1} | \chi\left (\frac{x - x_1}{d_{\min}^N/4} \right ) \nabla v_1(x)|^2 dx\\
&+ C \left (\int_{A_1} |\nabla v_1(x)|^2+\frac{1}{|d_{\min}^N|^2} \left| \nabla \chi\left (\frac{x - x_1}{d_{\min}^N/4} \right ) \right|^2 |v_1(x)|^2 \right ) dx.
\end{split}
\end{equation*}
Since $\chi\left (\frac{\cdot - x_1}{d_{\min}^N/4} \right)=1$ on $B(x_1, d_{\min}^N/4)$ we get 
\begin{align*}
\| \nabla v - \nabla v_1 \|^2_{L^2(\mathbb{R}^3 \setminus \bigcup_i\overline{ B_i})}  & \leq \| \nabla \tilde{v} \|^2_{L^2(\mathbb{R}^3 \setminus \bigcup_i \overline{B_i})} - \| \nabla v_1 \|^2_{L^2(\mathbb{R}^3 \setminus \bigcup_i \overline{B_i})}\,, \\
& \lesssim \int_{A_1} |\nabla v_1(x)|^2dx\,, \\
&+ \int_{A_1} \frac{1}{|d_{\min}^N|^2}\left | \nabla \chi\left (\frac{x - x_1}{d_{\min}^N/4} \right )\right|^2  |v_1|^2 dx,
\end{align*}
Thanks to \eqref{Stokeslet_maj} and \eqref{decay_rate} we have: 
\begin{equation*}
\begin{split}
& \int_{A_1}\frac{1}{|d_{\min}^N|^2} \left | \nabla \chi \left (\frac{x - x_1}{d_{\min}^N/4} \right )\right|^2 \left| v_1\right |^2 \\
&\lesssim \|\nabla \chi\|^2_{\infty} \int_{A_1} \frac{1}{|d_{\min}^N|^2} \left(R^2 \frac{|V|^2}{|x-x_1|^2}+ R^6 \frac{|D|^2}{|x-x_1|^4} \right)\,,\\
 & \lesssim \frac{1}{|d_{\min}^N|^2} \int_{d_{\min}^N/4}^{d_{\min}^N/2} \left ( R^2{|V|^2}+ R^6 \frac{|D|^2}{r^2} \right) dr\,,\\
& \lesssim \frac{1}{|d_{\min}^N|^2} \left ( R^2{|V|^2 d_{\min}^N}+ R^6 \frac{|D|^2}{d_{\min}^N} \right).
\end{split}
\end{equation*}
Reproducing an analogous computation for the first term we obtain finally
\begin{equation}
\| \nabla v - \nabla v_1 \|^2_{L^2(\mathbb{R}^3 \setminus \bigcup_i B_i)} \lesssim \frac{R^2}{{d_{\min}^N}} |V|^2 + \frac{R^6}{|d_{\min}^N|^{3}}|D|^2.
\end{equation}
This yields the expected result.\end{proof}
\subsection{Estimation of the fluid stresslet}
In this part we focus on approaching the stresslet $S_i$, $ 1 \leq i \leq N$, see \eqref{TS} for the definition. Unlike the drag force $F_i$ and torque $T_i$, the symmetric part of the first momentum does not appear in the ODEs governing the motion of particles, see \cite[Section 2.2.3]{Guazzelli&Morris} for more details. However, in order to approximate the velocities $(V_i,\Omega_i)$, we only need to estimate its value. Precisely we have 
\begin{proposition}\label{strain_estimate}
For $N$ sufficiently large, there exists a positive constant $C>0$ independent of the data such that we have for all $1 \leq i \leq N$
$$
|S_i| \lesssim  \frac{R^3}{|d_{\min}^N|^{3/2}}\underset{1 \leq j \leq N}{\max} \left(|V_j|+R|\Omega_j|\right)\,.
$$
\end{proposition}
\begin{proof}
We fix $i=1$. Let $E$ be a trace-free symmetric $3\times 3$ matrix. We define $v$ as the unique solution to the following Stokes equation
\begin{equation}
\left \{
\begin{array}{rcl}
-\Delta v + \nabla p & = &0, \\ 
\div v & = & 0,
\end{array}
\text { on $\mathbb{R}^3 \setminus {\underset{i=1}{\overset{N}{\bigcup}} \overline{B_i}}$, } 
\right.
\end{equation}
completed with the boundary conditions
\begin{equation}
\left \{
\begin{array}{rcl}
v & = &  E(x-x_1), \text{ on $ \partial B_1,$} \\
v & = & 0, \text{ on $\partial B_i$, $ i\neq 1,$} \\
\underset{|x| \to \infty}{\lim} |v(x)| & = & 0.
\end{array}
\right.
\end{equation}
We also denote by $(v_1,p_1)$ the special solution $(A^{(2)}_{x_1,R}[E],P^{(2)}_{x_1,R}[E])$. We have thanks to the symmetry of $E$
\begin{align}\label{formule___1}
S_1 : E&= \int_{\partial B_1} \sym \left( [\sigma(u^N,p^N) n ]\otimes(x-x_1) \right): E \,, \notag\\
&=-\int_{\partial B_1} [\sigma(u^N,p^N) n ] \cdot E(x-x_1) \,, \notag\\
&=- \int_{\partial B_1} [\sigma(u^N,p^N) n ] \cdot v \,, \notag\\
&=2 \int_{\mathbb{R}^3 \setminus \bigcup_i \overline{B_i}} D(u^N) : D(v)\,, \notag\\
&=2 \int_{\mathbb{R}^3 \setminus \bigcup_i \overline{B_i}} D(u^N) : D(v-v_1) +2 \int_{\mathbb{R}^3 \setminus \bigcup_i \overline{B_i}} D(u^N) : D(v_1) \,.
\end{align}
Using an integration by parts we have for the second term in the right hand side
\begin{align*}
\phantom{a} & 2 \int_{\mathbb{R}^3 \setminus \bigcup_i \overline{B_i}} D(u^N) : D(v_1)=- \underset{i=1}{\overset{N}{\sum}} \int_{\partial B_i} [\sigma(v_1,p_1) n ] \cdot (V_i + \Omega_i \times (x-x_i))\,,\\
&=  -\underset{i=1}{\overset{N}{\sum}} \left ( \int_{\partial B_i} [\sigma(v_1,p_1) n ] \right) \cdot V_i - \left ( \int_{\partial B_i} [\sigma(v_1,p_1) n ]\times (x-x_i) \right)\cdot \Omega_i\,,\\
&=0\,,
\end{align*}
since $v_1$ corresponds to a flow submitted only to a strain, see \eqref{stresslet_force_torque}. For the first term in the right hand side of \eqref{formule___1}, using Lemma \ref{approximation_result} we have 
$$
\left | \int_{\mathbb{R}^3 \setminus \bigcup_i \overline{B_i}} D(u^N) : D(v-v_1) \right| \leq  \| \nabla u^N\|_{L^2(\mathbb{R}^3 \setminus \bigcup_i \overline{B_i})} \frac{R^3}{|d_{\min}^N|^{3/2}}|E| \,.
$$
It remains to estimate $\| \nabla u^N\|_{L^2(\mathbb{R}^3 \setminus \bigcup_i \overline{B_i})}$. One can reproduce the same arguments as for the proof of Lemma \ref{approximation_result} or follow the same proof as \cite[Lemma 10]{Hillairet} to get
$$
 \| \nabla u^N\|_{L^2(\mathbb{R}^3 \setminus \bigcup_i \overline{B_i})}^2 \lesssim \, \underset{i}{\max} (|V_i|^2 + R^2 | \Omega_i|^2).
$$
Gathering all the estimates we obtain
$$
S_1 : E \lesssim \frac{R^3}{|d_{\min}^N|^{3/2}}|E|\, \underset{i}{\max} \left( |V_i|+ R| \Omega_i|  \right)\,,
$$
this being true for all symmetric trace-free $3\times3$ matrix $E$, we obtain the desired result.
\end{proof}
\section{Analysis of the stationary Stokes equation}
This Section is devoted to the analysis of a method of reflections and computation of the unknown velocities $(V_i,\Omega_i)_{1 \leq i \leq N}$. We remind that, for fixed time, $u^N$ is the unique solution to the stationary Stokes problem \begin{equation*}
\left \{
\begin{array}{rcl}
-\Delta u^N + \nabla p^N & = &0, \\ 
\div u^N & = & 0,
\end{array}
\text { on $\mathbb{R}^3 \setminus \underset{i=1}{\overset{N}{\bigcup}}\overline{B_i}$, } 
\right.
\end{equation*}
completed with the no-slip boundary conditions
\begin{equation*}
\left \{
\begin{array}{rcl}
u^N & = &  V_i +\Omega_i \times (x-x_i), \text{ on $ \partial {B_i}$}, \\
\underset{|x| \to \infty}{\lim} |u^N(x)| & = & 0,
\end{array}
\right.
\end{equation*}
where $(V_i,\Omega_i)$ are the unique velocities satisfying
\begin{eqnarray}\label{force}
F_i+ mg=0 \,,&  T_i= 0 \,,& \forall \,1 \leq i \leq N.
\end{eqnarray}
In this Section, we show that at each fixed time $t\geq0$, the convergence of the method of reflections toward the unique solution $u^N$ holds true in the case where $(X^N(t))_{N \in \mathbb{N}^*} \in \mathcal{X}(\bar{M}, \mathcal{E})$ and under the assumption that $r_0\bar{M}^{1/3}$ is small enough.
\subsection{The method of reflections} \label{reflection}
In this part, we present and prove the convergence of a modified method of reflections for the velocity field $u^N$ for arbitrary $N\in \mathbb{N}^*$, we remind that $u^N$ is the unique solution to the stationary Stokes problem \eqref{eq_stokes}, \eqref{cab_stokes}, 
with unique velocities $(V_i,\Omega_i)$ satisfying \eqref{force}.
The main idea is to express $u^N$ as the superposition of $N$  fields produced by the isolated $N$ particle. Thanks to the superposition principle, we know that the velocity field $$\underset{i=1}{\overset{N}{\sum}}\left (U_{x_j,R}[V_j](x)+A_{x_j,R}[\Omega_j](x)\right),$$ satisfies a Stokes equation on $\mathbb{R}^3 \setminus \underset{i}{ \bigcup \overline{B_i}} $. But this velocity field does not match the boundary conditions of $u^N$. Indeed, for all $1 \leq i \leq N$ and $x \in B_i$ we have
\begin{align*}
 u_*^{(1)}(x)& := u^N(x)- \underset{j=1}{\overset{N}{\sum}} \left (U_{x_j,R}[V_j](x)+A_{x_j,R}[\Omega_j](x)\right) \,,\\
 & = - \underset{i\neq j}{\overset{N}{\sum}} \left (U_{x_j,R}[V_j](x)+ A_{x_j,R}[\Omega_j](x)\right)\,,
\end{align*}
which represents the error committed on the boundary conditions when approaching $u^N$ by the sum of the particular Stokes solutions. In this paper, for all $u_* \in \mathcal{C}^\infty (\underset{i}{\bigcup}  \overline{B_i})$ we use the notation $U[u_*]$ to define the unique solution of the Stokes problem 
\begin{equation}
\left \{
\begin{array}{rcl}
-\Delta u + \nabla p & = &0, \\ 
\div u & = & 0,
\end{array}
\text { on $\mathbb{R}^3 \setminus \underset{i=1}{\overset{N}{\bigcup}} \overline{B_i}, $ } 
\right.
\end{equation}
completed by the boundary conditions
\begin{equation}
\left \{
\begin{array}{rcl}
u & = & u_*(x), \text{ on $  B_i,$} \\
\underset{|x| \to \infty}{\lim} |u(x)| & = & 0 ,
\end{array}
\right.
\end{equation}
hence, we can write 
$$
u^N= \underset{i=1}{\overset{N}{\sum}} U_{x_i,R}[V_i]+A_{x_j,R}[\Omega_j](x) + U[u_*^{(1)}].
$$
Note that the boundary condition $u_*^{(1)}$ is not constant on each particle $B_i$, thus, the idea is to approach $u_*^{(1)}$ by  
\begin{equation}\label{expansion}
 u_*^{(1)}(x) \sim u_*^{(1)}(x_i) + \nabla u_*^{(1)}(x_i) \cdot (x-x_i),
\end{equation}
on each particle $B_i$ and write $U[u_*^{(1)}]$ as follows:
$$
U[u_*^{(1)}]= \underset{j=1}{\overset{N}{\sum}} \left ( U_{x_j,R}[V_j^{(1)}] + A_{x_j,R}[\nabla_j^{(1)}] \right ) +U[u_*^{(2)}],
$$
where
$$ V_i^{(1)}:= u_*^{(1)}(x_i)= - \underset{j\neq i}{\sum} \left (U_{x_j,R}[V_j](x_i)+A_{x_j,R}[\Omega_j](x_i)\right),
$$
$$ \nabla_i^{(1)}:= \nabla u_*^{(1)}(x_i)= - \underset{j\neq i}{\sum}\left (\nabla U_{x_j,R}[V_j](x_i)+\nabla A_{x_j,R}[\Omega_j](x_i)\right),
$$
remark that $\nabla_i^{(1)}$ has null trace due to the fact that $$\div u_*^{(1)}(x_i)=0.$$
We set then $U[u_*^{(2)}]$ the new error term satisfying
$$
u^N= \underset{j=1}{\overset{N}{\sum}} \left ( U_{x_j,R}[V_j] + A_{x_j,R}[\Omega_j] \right )+ \underset{j=1}{\overset{N}{\sum}} \left ( U_{x_j,R}[V_j^{(1)}] + A_{x_j,R}[\nabla_j^{(1)}] \right )+U[u_*^{(2)}],
$$
where for all $1\leq i \leq N$, and $x\in B_i$ 
\begin{align*}
u_*^{(2)}(x)&= u_*^{(1)}(x)- \underset{j=1}{\overset{N}{\sum}} \left ( U_{x_j,R}[V_j^{(1)}](x) + A_{x_j,R}[\nabla_j^{(1)}](x) \right ) \,,\\
 &= u_*^{(1)}(x)- V_i^{(1)} - \nabla_i^{(1)} (x-x_i) - \underset{j\neq i }{\overset{N}{\sum}} \left ( U_{x_j,R}[V_j^{(1)}](x) + A_{x_j,R}[\nabla_j^{(1)}](x) \right ).
\end{align*}
We iterate then the process by setting for all $1\leq i \leq N$ 
\begin{eqnarray}\label{formule0}
V_i^{(0)}:= V_i\,,& \nabla_i^{(0)}:= \Omega_i,
\end{eqnarray}
and for $p\geq 1 $,
\begin{eqnarray}\label{formule1}
V_i^{(p)}:= u_*^{(p)}(x_i)\,,& \nabla_i^{(p)}:= \nabla u_*^{(p)}(x_i),
\end{eqnarray}
for the error term we set
\begin{equation}\label{formule3}
u_*^{(0)}(x):= \underset{i}{\overset{N}{\sum}} \left( V_i+\Omega_i\times (x-x_i) \right)\, 1_{ B_i},
\end{equation}
and define for all  $p\geq 0$, $1\leq i \leq N$, $x\in B_i$
\begin{equation}\label{formule2}
\begin{split}
u_*^{(p+1)}(x) &= u_*^{(p)}(x)- \underset{j=1}{\overset{N}{\sum}} \left ( U_{x_j,R}[V_j^{(p)}](x) + A_{x_j,R}[\nabla_j^{(p)}](x) \right ) \\
& = u_*^{(p)}(x) - u_*^{(p)}(x_i) - \nabla u_*^{(p)}(x_i)(x-x_i)\\
& - \underset{j\neq i }{\overset{N}{\sum}} \left ( U_{x_j,R}[V_j^{(p)}](x) + A_{x_j,R}[\nabla_j^{(p)}](x) \right ) .
\end{split}
\end{equation}
With this construction the following equality holds true for all $k\geq1$  
\begin{equation}\label{devlopment}
u^N= \underset{p=0}{\overset{k}{\sum}}  \underset{j=1}{\overset{N}{\sum}} \left ( U_{x_j,R}[V_j^{(p)}] + A_{x_j,R}[\nabla_j^{(p)}] \right ) +U[u_*^{(k+1)}].
\end{equation}
\begin{remark}\label{reflection_generalization}
This method of reflection is obtained by expanding the error term $u_*$ up to the first-order
$$u_*(x) = u_*(x_i)+ \nabla u_*(x_i) (x-x_i)+o \,(|x-x_i|^2),$$
which leads us to formula \eqref{devlopment}.
If one consider an expansion of $u_*$ up to the zeroth-order then one obtain only a stokeslet development:  
$$
u^N= \underset{p=0}{\overset{k}{\sum}}  \underset{j=1}{\overset{N}{\sum}} U_{x_j,R}[V_j^{(p)}] +U[u_*^{(k+1)}].
$$
The main difference between these two expansions is that the first one allows us to tackle the particle rotation. It also helps us to obtain a converging method of reflections for a more general assumption on the minimal distance.\\
We emphasize that we only need to show that the series $\left(\underset{p=0}{\overset{k}{\sum}} V_i^{(p)},\,\underset{p=0}{\overset{k}{\sum}}\nabla_i^{(p)}\right)_{k \in \mathbb{N}}$ for all $ 1 \leq i \leq N$ converge to obtain the convergence of the expansion \eqref{devlopment}, see Lemma \ref{cvg_serie} and Proposition \ref{conv_expansion}. Precisely, the only assumptions needed to obtain the convergence of the series are the smallness of $\bar{M}^{1/3}r_0$, assumption \eqref{bound_concentration} and the fact that
\begin{eqnarray*}
 \underset{N \to \infty}{\lim}  \frac{|\lambda^N|^3}{d_{\min}^N}=0 \,,& \underset{N \to \infty}{\lim} \frac{R|\lambda^N|^3}{|d_{\min}^N|^2}=0,
\end{eqnarray*}
which is less restrictive than \eqref{hyp1}.\\
The second step is to show that the expansion is a good approximation of the unique solution $u^N$. This is ensured by Proposition \ref{convergence_reflection}. Precisely, in addition of the previous assumptions, we need the following uniform bound
$$
\underset{N \in \mathbb{N}^*}{\sup} \frac{R |\lambda^N|^3}{|d_{\min}^N|^3} < + \infty. $$
One can show that this assumption is less restrictive than \eqref{hyp1} and allows us to consider smaller minimal distance.
To reach lower bound for the minimal distance, one may develop $u_*$ at higher orders.
\end{remark}
\subsubsection{Preliminary estimates}
Recall that the dependence in time is implicit in this Section. All the following estimates hold true under the assumption that there exists a sequence $(\lambda^N)_{N \in \mathbb{N}^*}$ and two positive constants $\bar{M}, \mathcal{E}$ such that $(X^N)_{N \in \mathbb{N}^*} \in \mathcal{X}(\bar{M}, \mathcal{E})$, see Definition \ref{def_regime} and $\bar{M}^{1/3}r_0$ is small enough.
\begin{lemma}\label{cvg_serie}
Assume that there exists $\bar{M}, \mathcal{E}$ and a sequence $(\lambda^N)_{N \in \mathbb{N}^*}$ such that the particle configuration $(X^N)_{N\in \mathbb{N}^*}$ lies in $\mathcal{X}(\bar{M}, \mathcal{E})$.
If $\bar{M}^{1/3} r_0$ is small enough, there exists a positive constant $K <1/2$ and $N(r_0,\bar{M},\mathcal{E})\in \mathbb{N}^*$  such that
$$
\underset{i}{\max}|V_i^{(p+1)}| + R \,  \underset{i}{\max}|\nabla_i^{(p+1)}| \leq K \left ( \underset{i}{\max}|V_i^{(p)}| + R \, \underset{i}{\max}|\nabla_i^{(p)}|\right ),
$$
for all $N\geq N(r_0,\bar{M},\mathcal{E})$.
\end{lemma}
\begin{proof}
Using formulas \eqref{formule1} and \eqref{formule2} we get
\begin{align}\label{formula_reflection}
V_i^{(p+1)}&=u_*^{(p+1)}(x_i) \notag \,,\\
& = - \underset{j\neq i }{\overset{N}{\sum}} \left ( U_{x_j,R}[V_j^{(p)}](x_i) + A_{x_j,R}[\nabla_j^{(p)}](x_i) \right ),
\end{align}
and
\begin{align}\label{formule_nabla}
\nabla_i^{(p+1)}&=\nabla u_*^{(p+1)}(x_i)\,, \notag \\
& = - \underset{j\neq i }{\overset{N}{\sum}} \left (\nabla U_{x_j,R}[V_j^{(p)}](x_i) +\nabla A_{x_j,R}[\nabla_j^{(p)}](x_i) \right ).
\end{align}
This yields, for all $ 1 \leq i \leq N$, using the decay-rate of the special solutions \eqref{decay_rate}, \eqref{Stokeslet_maj} and Lemma \ref{JO} with $k=1$ and $k=2$ 
\begin{equation*}
\begin{split}
&|V_i^{(p+1)}|\\
& \leq C  \underset{j\neq i }{\sum} R \,\frac{|V_j^{(p)}|}{d_{ij}} + R^3 \frac{|\nabla_j^{(p)}|}{d_{ij}^2} \\ 
& \leq C r_0\left( \underset{i}{\max}|V_i^{(p)}| + R\, \underset{i}{\max}|\nabla_i^{(p)}| \right ) \left ( \frac{|\lambda^N|^3}{|d_{\min}^N|}\bar{M}+ \bar{M}^{1/3}+ \frac{R|\lambda^N|^3}{|d_{\min}^N|^2} + R \bar{M}^{2/3}\right ),
\end{split}
\end{equation*}
similarly, using \eqref{Rdmin}, we have for all $1 \leq i \leq N$
\begin{align*}
|\nabla_i^{(p+1)}|& \leq C \underset{j\neq i }{\sum} R \,\frac{|V_j^{(p)}|}{d_{ij}^2} + R^3 \frac{|\nabla_j^{(p)}|}{d_{ij}^3}\,, \\ 
& \leq C \left( \underset{i}{\max}|V_i^{(p)}| + R\, \underset{i}{\max}|\nabla_i^{(p)}| \right ) \left ( \underset{j\neq i}{\sum} \frac{R}{d_{ij}^2} + \frac{1}{d_{\min}^N} \underset{j\neq i}{\sum} \frac{R^2}{d_{ij}^2} \right)\,,\\
& = C \left( \underset{i}{\max}|V_i^{(p)}| + R\, \underset{i}{\max}|\nabla_i^{(p)}| \right ) \left ( \underset{j\neq i}{\sum} \frac{R}{d_{ij}^2}\right) \left (1 + \frac{R}{d_{\min}^N} \right )\,, \\
& \leq C r_0 \left( \underset{i}{\max}|V_i^{(p)}| + R\, \underset{i}{\max}|\nabla_i^{(p)}| \right ) \left ( \frac{|\lambda^N|^3}{|d_{\min}^N|^2} \bar{M}+ \bar{M}^{2/3} \right )\,.
\end{align*}
Finally
\begin{align*}
\underset{i}{\max}|V_i^{(p+1)}| + R \,  \underset{i}{\max}|\nabla_i^{(p+1)}| &\leq  C r_0 \left( \underset{i}{\max}|V_i^{(p)}| + R\, \underset{i}{\max}|\nabla_i^{(p)}| \right )\\
& \times \left ( \frac{|\lambda^N|^3}{d_{\min}^N} \bar{M}+ \bar{M}^{1/3} + \frac{R|\lambda^N|^3}{|d_{\min}^N|^2}\bar{M} + R \bar{M}^{2/3} \right ) .
\end{align*}
For the second term on the right hand side we have
$$
\frac{|\lambda^N|^3}{d_{\min}^N}+\frac{R|\lambda^N|^3}{|d_{\min}^N|^2} = \frac{|\lambda^N|^3}{|d_{\min}^N|^2}\left(d_{\min}^N+ R \right)\,,
$$
which vanishes when $N$ tends to infinity according to \eqref{hyp1} and \eqref{compatibility}. Moreover, if $r_0 \bar{M}^{1/3} $ is small enough, this ensures the existence of a positive constant $K < 1/2 $  such that
$$\underset{i}{\max}|V_i^{(p+1)}| + R \,  \underset{i}{\max}|\nabla_i^{(p+1)}| \leq  K \left( \underset{i}{\max}|V_i^{(p)}| + R\, \underset{i}{\max}|\nabla_i^{(p)}| \right ).
$$
for $N$ large enough and depending on $r_0$, $\bar{M}$ and $\mathcal{E}$. \end{proof}
We have the following estimate.
\begin{proposition}\label{conv_expansion}
Let $(U_i)_{1 \leq i \leq N}$ be $N$ vectors of $\mathbb{R}^3$ and $(D_i)_{1 \leq i \leq N}$ be $N$ trace-free $3\times3$ matrices. There exists $N(r_0,\bar{M},\mathcal{E})\in \mathbb{N}^*$ such that for all $N\geq N(r_0,\bar{M},\mathcal{E})$ we have 
$$
\left \| \underset{i=1}{\overset{N}{\sum}} \left ( U_{x_i,R}[ U_i] + A_{x_i,R} [D_i] \right )  \right \|_{\dot{H}^1(\mathbb{R}^3\setminus \underset{l}{\cup}\overline{B}_l)} \lesssim \underset{1 \leq i\leq N}{\max} (|U_i|+R|D_i|).
$$
\end{proposition}
\begin{proof}
Considering only the stokeslet expansion we have
\begin{multline}\label{term_calcul}
\left \| \underset{i=1}{\overset{N}{\sum}} U_{x_i,R}[ U_i]   \right \|_{\dot{H}^1(\mathbb{R}^3\setminus \underset{l}{\cup}\overline{B}_l)}^2 =\\  \underset{i=1}{\overset{N}{\sum}}  \left \|U_{x_i,R}[ U_i]   \right \|_{\dot{H}^1(\mathbb{R}^3\setminus \underset{i}{\cup}\overline{B}_i)}^2 + \underset{i=1}{\overset{N}{\sum}} \underset{j \neq i}{\overset{N}{\sum}} \int_{\mathbb{R}^3\setminus \underset{l}{\cup}\overline{B}_l} \nabla U_{x_i,R}[ U_i]: \nabla U_{x_j,R}[ U_j]\,.
\end{multline}
The first term in the right hand side of \eqref{term_calcul} can be computed using the fact that $U_{x_i,R}[ U_i]=U_i$ on $\partial B_i$, $1 \leq i\leq N$ and  formula \eqref{stokeslet_force}
\begin{align*}
\phantom{a}&  \underset{i=1}{\overset{N}{\sum}}  \left \|U_{x_i,R}[ U_i]   \right \|_{\dot{H}^1(\mathbb{R}^3\setminus \underset{l}{\cup}\overline{B}_l)}^2\leq \underset{i=1}{\overset{N}{\sum}} \int_{\mathbb{R}^3\setminus \overline{B}_i} \nabla U_{x_i,R}[ U_i]: \nabla U_{x_i,R}[ U_i]\,,\\
&= - \underset{i=1}{\overset{N}{\sum}}  \int_{\partial B_i} \left[\sigma(U_{x_i,R}[ U_i],P_{x_i,R}[ U_i]) n \right]\cdot U_i \,,\\
&= \underset{i=1}{\overset{N}{\sum}} 6 \pi R |U_i|^2\,,\\
&\leq 6\pi r_0 \left(\underset{1\leq i \leq N}{\max}|U_i|\right)^2\,.
\end{align*}
For the second term in the right hand side of \eqref{term_calcul} we write for all $ i \neq j $
\begin{equation*}
\begin{split}
&\int_{\mathbb{R}^3\setminus \underset{l}{\cup}\overline{B}_l} \nabla U_{x_i,R}[ U_i]: \nabla U_{x_j,R}[ U_j]\\
&= -\underset{l=1}{\overset{N}{\sum}} \int_{\partial B_l}\left[\sigma(U_{x_i,R}[ U_i],P_{x_i,R}[ U_i]) n \right]\cdot U_{x_j,R}[ U_j]\,,\\
&\leq \underset{l=1}{\overset{N}{\sum}} 4\pi R^2 \left\|\sigma(U_{x_i,R}[ U_i],P_{x_i,R}[ U_i]) \right\|_{L^\infty(\partial B_l)} \left\| U_{x_j,R}[ U_j]\right \|_{L^\infty(\partial B_l)}\,,\\
&:= \underset{l=1}{\overset{N}{\sum}} 4\pi R^2 \mathcal{O}_{i,j}^l.
\end{split}
\end{equation*}
According to the decay properties of the stokeslet \eqref{Stokeslet_maj} we have
\begin{eqnarray}\label{formule_decay1}
\left\|\sigma(U_{x_i,R}[ U_i],P_{x_i,R}[U_i]) \right\|_{L^\infty(\partial B_l)} &\lesssim &R\frac{|U_i|}{d_{il}^2} (1-\delta_{il})+ \frac{|U_i|}{R} \delta_{il} \,,\notag\\
\left\|U_{x_j,R}[U_j] \right\|_{L^\infty(\partial B_l)}& \lesssim & \frac{R|U_j|}{d_{jl}}(1-\delta_{jl})+ |U_j| \delta_{jl}\,,
\end{eqnarray}
where $\delta_{ij}$ is the Kronecker symbol. On the other hand, we recall that the triangle inequality $d_{ij} \leq d_{il}+ d_{jl}$ yields for all $ i\neq j \neq l$
\begin{equation}\label{triangle}
\frac{1}{d_{il} d_{jl}} \leq \frac{1}{d_{ij}} \left( \frac{1}{d_{il}}+\frac{1}{d_{jl}}  \right)\,.
\end{equation}
Using \eqref{formule_decay1}, formula \eqref{triangle} twice and Lemma \ref{JO} we obtain for all $ i\neq j$
\begin{align*}
\underset{l=1}{\overset{N}{\sum}} 4\pi R^2 \mathcal{O}_{i,j}^l&=\underset{l\neq i,j}{\sum} 4\pi R^2 \mathcal{O}_{i,j}^l + 4 \pi R^2 \mathcal{O}_{i,j}^i+4 \pi R^2\mathcal{O}_{i,j}^j\,,\\
& \lesssim \underset{l\neq i,j}{\sum}  R^2 \frac{R|U_i|}{d_{il}^2} \frac{R|U_j|}{d_{jl}} +  R^2 \frac{|U_i|}{R}\frac{R|U_j|}{d_{ij}} +  R^2 \frac{R|U_i|}{d_{ij}^2} |U_j|\,,\\
&\lesssim   \frac{R^3}{d_{ij}} \left( \frac{1}{d_{ij}}\left(\underset{l\neq i,j}{\sum}\frac{R}{d_{il}}+ \underset{l\neq i,j}{\sum}\frac{R}{d_{jl}} \right) + \underset{l\neq i,j}{\sum} \frac{R}{d_{il}^2}  \right)|U_j| |U_i| \\
& +R^2  \frac{|U_j||U_i|}{d_{ij}} +  R^3 \frac{|U_i||U_j|}{d_{ij}^2} \,,\\
&\lesssim  \frac{R^3}{d_{ij}} \left( \frac{1}{d_{ij}}\left( \frac{|\lambda^N|^3}{d_{\min}^N} \bar{M}+ \bar{M}^{1/3} \right) +\frac{|\lambda^N|^3}{|d_{\min}^N|^2} \bar{M}+\bar{M}^{2/3}  \right) |U_j| |U_i| \\
& +  R^2  \frac{|U_j||U_i|}{d_{ij}} + R^3 \frac{|U_i||U_j|}{d_{ij}^2} \,,\\
&\lesssim  \frac{R^3}{d_{ij}} \left(  \mathcal{E} \bar{M}+ \frac{\bar{M}^{1/3}r_0^2}{d_{\min}^N}+\bar{M}^{2/3}  \right) |U_j| |U_i| +  R^2  \frac{|U_j||U_i|}{d_{ij}} + R^3 \frac{|U_i||U_j|}{d_{ij}^2} \,,\\
&\lesssim \left[ \frac{R^3}{d_{ij}} \frac{1}{d_{\min}^N} +    \frac{R^2}{d_{ij}} +  \frac{R^3 }{d_{ij}^2}\right] |U_j||U_i| \,, \\
\end{align*}
where we kept only the largest terms using the fact that $d_{\min}^N $ vanishes according to \eqref{compatibility} for $N$ large enough. Hence, the second term in the right hand side of \eqref{term_calcul} yields using Lemma \ref{JO}
\begin{align*}
\phantom{a} & \underset{i=1}{\overset{N}{\sum}} \underset{j \neq i}{\overset{N}{\sum}} \int_{\mathbb{R}^3\setminus \underset{l}{\cup}\overline{B}_l} \nabla U_{x_i,R}[ U_i]: \nabla U_{x_j,R}[ U_j] \lesssim \underset{i=1}{\overset{N}{\sum}} \underset{j \neq i}{\overset{N}{\sum}} \left[ \frac{R^3}{d_{ij}}  \frac{1}{d_{\min}^N} +    \frac{R^2}{d_{ij}} +  \frac{R^3 }{d_{ij}^2}\right] |U_j||U_i|\,, \\
&\lesssim \underset{1\leq i \leq N}{\max}\left(\underset{j \neq i}{\overset{N}{\sum}} \left[ \frac{R^2}{d_{ij}} \frac{1}{d_{\min}^N}+ \frac{R}{d_{ij}} +  \frac{R^2 }{d_{ij}^2}\right] \right) \left(\underset{1\leq i \leq N}{\max}|U_i|\right)^2\,,\\
&\lesssim  \left[ \left( \frac{R}{d_{\min}^N} +1\right)\left(\frac{|\lambda^N|^3}{d_{\min}^N}\bar{M}+\bar{M}^{1/3} \right)+\frac{R|\lambda^N|^3}{|d_{\min}^N|^2} \bar{M}+R \bar{M}^{2/3} \right]\left(\underset{1\leq i \leq N}{\max}|U_i|\right)^2\,,\\
&\lesssim \left(\underset{1\leq i \leq N}{\max}|U_i|\right)^2\,,
\end{align*}
where we used the fact that $\frac{R}{d_{\min}^N} \lesssim1$ thanks to \eqref{Rdmin} and $ \frac{|\lambda^N|^3}{d_{\min}^N} \leq \frac{|\lambda^N|^3}{|d_{\min}^N|^2} d_{\min}^N\lesssim 1 $ according to \eqref{hyp1} and \eqref{compatibility}.  The term involving rotating and straining solutions $A_{x_i,R}[D_i]$ is handled analogously. 
\end{proof}
Since the series $\underset{p=0}{\overset{k}{\sum}} V_i^{(p)}$, $\underset{p=0}{\overset{k}{\sum}} \nabla_i^{(p)}$ are convergent, we denote their limit by
\begin{eqnarray*}
V_i^\infty := \underset{p=0}{\overset{\infty}{\sum}} V_i^{(p)} ,& \nabla_i^\infty:= \underset{p=0}{\overset{\infty}{\sum}} \nabla_i^{(p)}. \end{eqnarray*}
Thanks to the linearity of the Stokes equation and Proposition \ref{conv_expansion}, the expansion term 
$$
\underset{i=1}{\overset{N}{\sum}} \left ( U_{x_i,R} \left [\underset{p=0}{\overset{k}{\sum}}V_i^{(p)} \right] + A_{x_i,R} \left[\underset{p=0}{\overset{k}{\sum}} \nabla_i^{(p)}  \right] \right ),
$$
converges in $\dot{H}^1(\mathbb{R}^3 \setminus \underset{l}{\cup}\overline{B}_l)$ uniformly in $N$ to the expansion where we replace the series by their limit. This shows that the error term 
$U[u_*^{(k)}]$ has a limit when $k \to \infty$. In order to quantify the error term, we begin by the following estimate
\begin{proposition}\label{estimation_RM}
For all $k\geq1$ we set
$$
\eta^{(k)} :=  \underset{j}{\max} |V_j^{(k)}| + R\  \underset{j}{\max} |\nabla_j^{(k)}|.
$$
Under the same assumptions as Lemma \ref{cvg_serie}, there exists $N(r_0,\bar{M},\mathcal{E})\in \mathbb{N}^*$ such that for all $N\geq N(r_0,\bar{M},\mathcal{E})$ and $1\leq i \leq N$  
\begin{eqnarray*}
 \| \nabla^2 u_*^{(k+1)} \|_{L^\infty( B_i)} &\lesssim&   \left( 1+\frac{|\lambda^N|^3}{|d_{\min}^N|^3}+ |\log (\bar{M}^{1/3} \lambda^N)| \right) \underset{i}{\max}( |V_i|+ R |\Omega_i|),\\
 \|\nabla u_*^{(k+1)}\|_{L^\infty(B_i)} &\lesssim & R  \| \nabla^2 u_*^{(k+1)} \|_{L^\infty( B_i)} + \eta^{(k)},\\
 \| u_*^{(k)}\|_{L^\infty(B_i)} &\lesssim &R^2 \| \nabla^2 u_*^{(k+1)} \|_{L^\infty( B_i)} +  \eta^{(k)}.
\end{eqnarray*}
\end{proposition}
\begin{proof}
\text{} \\
\textbf{1. Estimate of $\| \nabla^2 u_*^{(k+1)} \|_{L^\infty(B_i)}$} \\
Let $x \in  B_i$, using formula \eqref{Rdmin} we recall that for $i\neq j $
$$
|x-x_j| \geq |x_i-x_j| - |x-x_i| \geq \frac{1}{2} d_{ij}.
$$
Applying this, formula \eqref{Rdmin}, the decay properties of the second gradient of single particle solutions \eqref{Stokeslet_maj_2}, \eqref{decay_rate_2} and the iteration formula \eqref{formule2} together with Lemma \ref{JO} for $k=3$ yields  
\begin{equation*}
\begin{split}
&|\nabla^2 u_*^{(k+1)}(x)|\\
&\leq | \nabla^2 u_*^{(k)}(x)| + \underset{j\neq i }{{\sum}}  |\nabla ^2 U_{x_j,R}[V_j^{(k)}](x)| + |\nabla^2 A_{x_j,R}[\nabla_j^{(k)}](x) | \,,\\
& \lesssim \| \nabla^2 u_*^{(k)} \|_{L^\infty(  B_i)} + \underset{j\neq i }{{\sum}} \frac{|V_j^{(k)}|}{d_{ij}^3}R+ \frac{|\nabla_j^{(k)}|}{d_{ij}^4}R^3\,, \\
& \lesssim  \| \nabla^2 u_*^{(k)} \|_{L^\infty(  B_i)} + \left ( \underset{j\neq i}{\sum} \frac{R}{d_{ij}^3} + \frac{R}{d_{\min}^N} \underset{j\neq i}{\sum} \frac{R}{d_{ij}^3} \right ) \left( \underset{j}{\max} |V_j^{(k)}| + R \underset{j}{\max} |\nabla_j^{(k)}| \right )\,, \\
&= \| \nabla^2 u_*^{(k)} \|_{L^\infty(  B_i)} + \left (\underset{j\neq i}{\sum} \frac{R}{d_{ij}^3}  \right) \left(1+ \frac{R}{d_{\min}^N} \right)
\eta^{(k)}\,, \\
& \lesssim  \| \nabla^2 u_*^{(k)} \|_{L^\infty(B_i)} + r_0 \bar{M} \left ( \frac{|\lambda^N|^3}{|d_{\min}^N|^3}+ |\log (\bar{M}^{1/3} \lambda^N)|+1
\right)\eta^{(k)},
\end{split}
\end{equation*}
hence, we iterate the formula and use the fact that $\nabla^2 u_*^{(0)}=0$ according to formula \eqref{formule3} to get
$$
 \| \nabla^2 u_*^{(k+1)} \|_{L^\infty( B_i)} \lesssim  \left (1+\frac{|\lambda^N|^3}{|d_{\min}^N|^3}+ |\log (\bar{M}^{1/3} \lambda^N)|\right) \underset{p=0}{\overset{k}{\sum}} \eta^{(p)},
$$
which yields the expected result by applying Lemma \ref{cvg_serie}.\\
\textbf{2. Estimate of $\| \nabla u_*^{(k+1)} \|_{L^\infty(B_i)}$}
\text{} \\
Let $x \in  B_i$, again, the decay properties of the gradient of the special solutions \eqref{Stokeslet_maj}, \eqref{decay_rate}, formula \eqref{formule2} and Lemma \ref{JO} yields 
\begin{equation*}
\begin{split}
&|\nabla u_*^{(k+1)}(x)|\\
& \leq |\nabla u_*^{(k)}(x) - \nabla u_*^{(k)}(x_i)| +  \underset{j\neq i }{{\sum}}  |\nabla  U_{x_j,R}[V_j^{(k)}](x_i)| + |\nabla A_{x_j,R}[\nabla_j^{(k)}](x_i) | \,,\\
&\lesssim R \|\nabla^2 u_*^{(k)}\|_{L^\infty(B_i)}+ \underset{j\neq i }{{\sum}} \frac{|V_j^{(k)}|}{d_{ij}^2}R + \frac{|\nabla_j^{(k)}|}{d_{ij}^3}R^3\,, \\
& \lesssim  R \|\nabla^2 u_*^{(k)}\|_{L^\infty(B_i)}+ \left (  \underset{j\neq i}{\sum} \frac{R}{d_{ij}^2} + \frac{R}{d_{\min}^N} \underset{j\neq i}{\sum} \frac{R}{d_{ij}^2} \right ) \left( \underset{j}{\max} |V_j^{(k)}| + R \underset{j}{\max} |\nabla_j^{(k)}| \right )\,, \\
& \lesssim  R \|\nabla^2 u_*^{(k)}\|_{L^\infty(B_i)} + \left( 1 +\frac{R}{d_{\min}^N} \right) \left (  \frac{|\lambda^N|^3}{|d_{\min}^N|^2}\bar{M}+ \bar{M}^{2/3} \right )r_0 \eta^{(k)},
\end{split}
\end{equation*}
again, according to \eqref{Rdmin}, note that for $N$ large enough, $1+\frac{R}{d_{\min}^N} \leq 2
$. We conclude using assumption \eqref{hyp1} to bound the right hand side by $\eta^{(k)}$ up to constants depending on $\bar{M}$, $\mathcal{E}$, $r_0$.\\
\textbf{3. Estimate of $\| u_*^{(k+1)} \|_{L^\infty(B_i)}$}
\text{} \\
Let $x \in  B_i$, again, the decay property \eqref{Stokeslet_maj}, \eqref{decay_rate} and formula \eqref{formule2} yields 
\begin{equation*}
\begin{split}
&|u_*^{(k+1)}(x)| \\
& \leq R^2 \|\nabla^2 u_*^{(k)}\|_{L^\infty(B_i)} +  \underset{j\neq i }{{\sum}}  |  U_{x_j,R}[V_j^{(k)}](x)| + | A_{x_j,R}[\nabla_j^{(k)}](x) |\,, \\
& \lesssim R^2 \|\nabla^2 u_*^{(k)}\|_{L^\infty(B_i)} + \underset{j\neq i }{{\sum}} \frac{|V_j^{(k)}|}{d_{ij}}R + \frac{|\nabla_j^{(k)}|}{d_{ij}^2}R^3\,, \\
& \lesssim R^2 \|\nabla^2 u_*^{(k)}\|_{L^\infty(B_i)}  +r_0 \left ( \frac{|\lambda^N|^3}{d_{\min}^N}\bar{M}+\bar{M}^{1/3}+\frac{R|\lambda^N|^3}{|d_{\min}^N|^2}\bar{M}+R\bar{M}^{2/3}\right ) \eta^{(k)}.
\end{split}
\end{equation*}
Using \eqref{hyp1} and \eqref{compatibility}, the right hand side can be bounded by $\eta^{(k)}$ up to constants depending on $\bar{M}$, $\mathcal{E}$, $r_0$.
\end{proof}
\subsubsection{Approximation result}\label{reflection_method}
We can now state the main result of this Section. 
%
\begin{proposition}\label{convergence_reflection}
Assume that there exists $\bar{M}, \mathcal{E}$ and a sequence $(\lambda^N)_{N \in \mathbb{N}^*}$ such that $(X^N)_{N\in \mathbb{N}^*} \in \mathcal{X}(\bar{M}, \mathcal{E})$. Assume moreover that $\bar{M}^{1/3} r_0$ is small enough.\\
There exists a positive constant $C=C({r}_0,\bar{M}, \mathcal{E})$ and $N(r_0,\bar{M},\mathcal{E})\in \mathbb{N}^*$ satisfying for all $N \geq N(r_0,\bar{M},\mathcal{E})$
$$
\underset{k\to \infty}{{\lim}} \| \nabla U[u_*^{(k+1)}]\|_{L^2(\mathbb{R}^3 \setminus \bigcup B_i)} \leq C  R \, \underset{i}{\max}\,( |V_i|+ R |\Omega_i|), 
$$
\end{proposition}
\begin{proof}
The aim is to estimate $\| \nabla U[u_*^{(k+1)}]\|_{L^2(\mathbb{R}^3 \setminus \bigcup B_i)}$. To this end, we construct a suitable extension $E[u_*^{(k+1)}]$ of the boundary conditions of $u_*^{(k+1)}$ and apply the variational principle \eqref{variational_form}. \vspace{1pt}
By construction, $u_*^{(k+1)}$ is regular and well defined on each particle $ B(x_i,R)$. Hence, we construct the extension piecewise in each $B(x_i,2R)$.
Let $1 \leq i\leq N$, for all $x\in B(x_i,2R)$ we set $$v^i(x):=u_1^{(i)}(x)+u_2^{(i)}(x),$$ 
where the first term $u_1^{(i)}$ matches the boundary condition on $B(x_i,R)$ and vanishes outside $B(x_i,2R)$. The second term is the correction needed to get $\div v_i =0$.
In order to obtain an extension of $u_*^{(k)}$ on $B(x_i,2R)$ we set 
\begin{eqnarray*}
u_1^{(i)}(x) & = & u_*^{(k)}\left(x_i+ R \frac{x-x_i}{|x-x_i|}\right) \chi \left(\left|\frac{x-x_i}{R}\right| \right), \text { if $ |x-x_i| \geq R,$} \\
u_1^{(i)}(x) & = & u_*^{(k)}(x), \text { if $ x \in B(x_i,R)$ } ,
\end{eqnarray*}
with $\chi$ a truncation function such that $\chi = 1$ on $[0,1]$ and $\chi=0$ outside $[0,2]$.\\
We have then 
\begin{equation}
\| \nabla u_1^{(i)}\|_{L^\infty(B(x_i,2R))} \leq K_\chi \left (\|\nabla u_*^{(k)}\|_{L^\infty(B(x_i,R))} + \frac{1}{R} \|u_*^{(k)}\|_{L^\infty(B(x_i,R))} \right ).
\end{equation}
In what follows we introduce the notation $A(x,r,R):= B(x,R) \setminus \overline{B(x,r)}$ for $r<R$. 
For the second term we set:
$$u_i^{(2)}= \mathcal{B}_{x_i\,,\, R\,,\,2R}(- \div u_i^{(1)}),$$
where $\mathcal{B}$ is the Bogovskii operator, see \cite[Appendix A Lemma 15 and 16]{Hillairet} for more details. \\
The construction satisfies: 
\begin{itemize}
\item $\text{ supp } u_i^{(2)} \subset A(x_i,R,2R)$ 
\item $ \div v_i= 0$
\item $v_i= u_i^{(1)}=u_*^{(k)}$ on $B(x_i,R)$
\end{itemize}
We set then
$$
E[u_*^{(k+1)}] = \underset{i}{\overset{N}{\sum}} v^i(x)1_{B(x_i,2R)},
$$
and thanks to the variational formulation we have 
\begin{align*}
\| \nabla U[u_*^{(k+1)}]\|_{L^2(\mathbb{R}^3 \setminus \bigcup B_i)}^2 &\leq \| \nabla E[u_*^{(k+1)}]\|_{L^2(\mathbb{R}^3 \setminus \bigcup B_i)}^2 \,,\\
& = \underset{i}{\overset{N}{\sum}}  \|\nabla v_i \|_{L^2(A(x_i,R,2R))}^2, \end{align*}
where we used the fact that the $v_i$ have disjoint support. \\
Thanks to the properties of the Bogovskii operator $\mathcal{B}_{x_i\,,\, R\,,\,2R}$ we get
\begin{align*}
 \|\nabla v_i \|_{L^2(B(x_i,R))}^2 & \lesssim \int_{A(x_i,R,2R)} |\nabla u_1^{(i)}|^2\,, \\
& \lesssim R^3  \|\nabla u_1^{(i)}\|^2_{L^\infty(B(A(x_i,R,2R)))}\,,\\
& \lesssim R^3 \left (  \|\nabla u_*^{(k)}\|_{L^\infty( B_i)}+ \frac{1}{R} \| u_*^{(k)}\|_{L^\infty( B_i)} 
\right )^2.
\end{align*}
Finally
\begin{align*}
\| \nabla U[u_*^{(k+1)}]\|_{L^2(\mathbb{R}^3 \setminus \bigcup B_i)}^2 \lesssim  \underset{i=1}{\overset{N}{\sum}} R^3 \left (  \|\nabla u_*^{(k)}\|_{L^\infty( B_i)}+ \frac{1}{R} \| u_*^{(k)}\|_{L^\infty( B_i)} \right )^2.
\end{align*}
Thanks to Proposition \ref{estimation_RM} we have
\begin{equation*}
\begin{split}
&\|\nabla u_*^{(k)}\|_{L^\infty(B_i)}+ \frac{1}{R} \| u_*\|_{L^\infty( B_i)}\\
 & \lesssim  \underset{i}{\max}(|V_i|+R|\Omega_i|)\,\left(R+ \frac{R |\lambda^N|^3}{|d_{\min}^N|^3}+ R | \log(\bar{M}^{1/3}\lambda^N)| \right)+  \left ( \frac{1}{R}+1 \right) \eta^{(k)}.
\end{split}
\end{equation*}
Since $$\eta^{(k)} \leq K^k \underset{i}{\max}(|V_i|+R|\Omega_i|),$$ with $K<1$ according to Lemma \ref{cvg_serie}, we get
\begin{multline*}
\| \nabla U[u_*^{(k+1)}]\|_{L^2(\mathbb{R}^3 \setminus \bigcup B_i)}^2\\
  \lesssim \underset{i}{\max}(|V_i|+R|\Omega_i|)^2 \,\Big \{ R \left(R+ \frac{R |\lambda^N|^3}{|d_{\min}^N|^3}+ R | \log(\bar{M}^{1/3}\lambda^N)| \right) + \left ( 1+R \right)   K^{k} \Big \}^2.
\end{multline*}
Since $K<1$ for $N$ large enough the second term on the right hand side, which is uniformly bounded with respect to $N$, vanishes when $k\to \infty$. This yields 
\begin{multline*}
\underset{k\to \infty}{\lim} \| \nabla U[u_*^{(k+1)}]\|_{L^2(\mathbb{R}^3 \setminus \bigcup B_i)}\\ \lesssim  R\, \underset{i}{\max}(|V_i|+R|\Omega_i|)\,\left(R+ \frac{R |\lambda^N|^3}{|d_{\min}^N|^3}+ R | \log(\bar{M}^{1/3}\lambda^N)| \right).
\end{multline*}
The second term on the right hand side can be bounded using assumptions \eqref{hyp1}, \eqref{hyp3} and \eqref{Rdmin}
$$
R+ \frac{R |\lambda^N|^3}{|d_{\min}^N|^3}+ R | \log(\bar{M}^{1/3}\lambda^N)|\lesssim R+ \frac{R}{d_{\min}^N} \mathcal{E} +\frac{|\log \bar{M}|+ \log N }{N}\lesssim 1,
$$
Finally we obtain
\begin{equation*}
\underset{k\to \infty}{{\lim}} \| \nabla U[u_*^{(k+1)}]\|_{L^2(\mathbb{R}^3 \setminus \bigcup B_i)} \lesssim R\, \underset{i}{\max}(|V_i|+R|\Omega_i|)\,,
\end{equation*}
which is the desired result.
\end{proof}
\begin{remark}\label{u_*_infty}
According to Proposition \ref{estimation_RM} we have for all $1 \leq i \leq N$ 
\begin{align*}
\|u_*^{(k+1)}\|_{L^\infty(B_i)} &\lesssim R^2 \|\nabla^2  u_*^{(k+1)}\|_{L^\infty(B_i)} + \eta^{(k)}\\
& \lesssim \underset{i}{\max}(|V_i|+R|\Omega_i|) \,\left \{ R \left (\frac{R |\lambda^N|^3}{|d_{\min}^N|^3}+ R | \log(\bar{M}^{1/3}\lambda^N)|\right)+  K^{k} \right \}.
\end{align*}
as for the proof of Proposition \ref{convergence_reflection} the second term vanishes when $k \to \infty$ and we obtain 
\begin{equation*}
\underset{k\to \infty}{\overline{\lim}} \|u_*^{(k+1)}\|_{L^\infty(B_i)} \lesssim \underset{i}{\max}(|V_i|+R|\Omega_i|) R.
\end{equation*}
\end{remark}
\subsubsection{Some associated estimates}
We recall that we aim to compute the velocities $(V_i,\Omega_i)$ associated to the unique solution $u^N$ of the Stokes equation: 
\begin{equation*}
\left \{
\begin{array}{rcl}
-\Delta u^N + \nabla p^N & = &0, \\ 
\div u^N & = & 0,
\end{array}
\text { on $\mathbb{R}^3 \setminus \underset{i=1}{\overset{N}{\bigcup}}\overline{B_i}$, } 
\right.
\end{equation*}
completed with the no-slip boundary conditions
\begin{equation*}
\left \{
\begin{array}{rcl}
u^N & = &  V_i +\Omega_i \times (x-x_i), \text{ on $ \partial {B_i}$}, \\
\underset{|x| \to \infty}{\lim} |u^N(x)| & = & 0,
\end{array}
\right.
\end{equation*}
with 
\begin{eqnarray*}
F_i+ mg=0\,,&  T_i= 0 \,,& \forall 1 \leq i \leq N.
\end{eqnarray*}
The method of reflections obtained in this Section helps us to describe the velocity field $u^N$ in terms of explicit flows
$$
u^N = \underset{j=1}{\overset{N}{\sum}} \left ( U_{x_j,R}\left[ V_j^{(\infty)}\right] + A_{x_j,R}\left[ \nabla_j^{(\infty)}\right] \right )+  \underset{k \to \infty}{{\lim}}  U[u_*^{(k)}].
$$
In order to extract a formula for the unknown velocities $(V_i,\Omega_i)$, $1 \leq i \leq N$ we need to compute first the velocities $V_i^{(\infty)}$ and matrices $\nabla_i^{(\infty)}$. Applying the method of reflections and writing the force, torque and stresslet associated to the unique solution $u^N$ in two different ways we get the following result.
\begin{lemma}\label{asymptotic_velocity}
Consider the same assumptions as Proposition \ref{convergence_reflection}.\\
There exists $N(r_0,\bar{M},\mathcal{E}) \in\mathbb{N}^*$ such that for all $N\geq N(r_0,\bar{M},\mathcal{E})$
\begin{eqnarray*}
V_i^\infty =& \kappa g + O\left( \underset{i}{\max}(|V_i|+ R|\Omega_i|) \frac{R}{\sqrt{d_{\min}^N}} \right)\,,&1 \leq i \leq N\,,\\
R\left|\nabla_i^\infty \right|=& O\left( \underset{i}{\max}(|V_i|+ R|\Omega_i|) \frac{R}{\left|d_{\min}^N\right|^{3/2}}\right)\,,&1 \leq i \leq N\,,
\end{eqnarray*}
where $\kappa g$ is defined thanks to formula \eqref{kappag}.
\end{lemma}
\begin{proof}
For the sake of clarity we fix $i=1$ and the same result holds for all $1\leq i \leq N$.\\ Let $V \in \mathbb{R}^3$, $D$ a trace-free $3 \times 3$ matrix. 

The main idea is to apply an integration by parts with a suitable test function $v\in D_\sigma(\mathbb{R}^3 \setminus \underset{i}{\bigcup} \overline{B_i})$ such that $v = V +D (x-x_1)$ on $\partial B_1$ and $v=0$ on the other $\partial B_j$, $j \neq 1$.
We choose $v$ the unique solution to the Stokes equation:
\begin{equation}
\left \{
\begin{array}{rcl}
-\Delta v + \nabla p & = &0, \\ 
\div v & = & 0,
\end{array}
\text { on $\mathbb{R}^3 \setminus \underset{i=1}{\overset{N}{\bigcup}} \overline{B_i}$, } 
\right.
\end{equation}
completed by the boundary conditions
\begin{equation}
\left \{
\begin{array}{rcl}
v & = &  V+D(x-x_1), \text{ on $ \partial B_1,$} \\
v & = & 0 \text{ on $\partial B_i$, $ i\neq 1,$} \\
\underset{|x| \to \infty}{\lim} |v(x)| & = & 0.
\end{array}
\right.
\end{equation}
We extend $u^N$ and $v$ by their boundary values on all $B_i$, $1\leq i \leq N$. We set $E =\sym(D)$, $\Omega= \asym(D)$. An integration by parts yields
\begin{align}
2 \int_{\mathbb{R}^3\setminus \underset{i}{\bigcup} B_i} D(u^N): D(v)&= -\underset{i}{\sum}\int_{\partial B_i} \left [\sigma (u^N,p^N) n \right ] \cdot v  \,,\nonumber\\ 
& =- \int_{\partial B_1} \left [\sigma (u^N,p^N) n \right ] \cdot (V + \Omega \times (x-x_1) + E (x-x_1))\,,\nonumber \\ 
& = -V \cdot \int_{\partial B_i} \sigma (u^N,p^N) n \\
&- \Omega \cdot \int_{\partial B_i} (x-x_i) \times \left [ \sigma (u^N,p^N) n \right ]\,, \nonumber \\ 
&-  E : \int_{\partial B_i} (x-x_i) \otimes \left [ \sigma (u^N,p^N) n \right ] \,, \nonumber \\ 
&= -V \cdot F_1 - \Omega \cdot T_1-E :S_1, \label{egalite1}
\end{align}
see \eqref{FM} and \eqref{TS} for the definition of the force $F_1$, torque $T_1$ and stresslet $S_1$.
On the other hand, we apply the method of reflections  to get
\begin{multline}\label{egalite2}
\int_{\mathbb{R}^3 \setminus \underset{i}{\bigcup}B_i} D(u^N) : D(v)= \\  \underset{j=1}{\overset{N}{\sum}} \int_{\mathbb{R}^3 \setminus \underset{i}{\bigcup}B_i} (D( U_{x_j,R}[V_j^\infty]) + D(\nabla A_{x_j,R}[\nabla_j^\infty])) : D (v)\\ + \underset{k \to \infty}{{\lim}} \int_{\mathbb{R}^3 \setminus \underset{i}{\bigcup}B_i} D(U[u_*^{k}]) : D(v).
\end{multline}
For the first term we integrate by parts to get for all $1 \leq j \leq N$
\begin{eqnarray*}
2\int_{\mathbb{R}^3 \setminus \underset{i}{\bigcup}B_i}  D (U_{x_j,R}[V_j^\infty]) : D(v) &=& -\underset{i=1}{\overset{N}{\sum}}\int_{\partial B_i} \left [\sigma (U_{x_j,R}[V_j^\infty],P_{x_j,R}[V_j^\infty]) n \right ]\cdot v. \\
2\int_{\mathbb{R}^3 \setminus \underset{i}{\bigcup}B_i}  D (A_{x_j,R}[\nabla_j^\infty]) : D(v) &=& -\underset{i=1}{\overset{N}{\sum}} \int_{\partial B_i} \left [\sigma (A_{x_j,R}[\nabla_j^\infty],P_{x_j,R}[\nabla_j^\infty]) n \right ]\cdot v.
\end{eqnarray*}
Recall that $v$ vanishes on $\partial B_i$, $i\neq 1$ and hence, the sums above are reduced to the first term. Applying \eqref {stresslet_strain} \eqref {rotelet_torque} and \eqref{stokeslet_force} there holds for all $1 \leq j \leq N$
$$
\int_{\partial B_1} \left [\sigma (U_{x_j,R}[V_j^\infty],P_{x_j,R}[V_j^\infty]) n \right ]\cdot v  =  -6 \pi R V_1^\infty \cdot V \delta_{1j} .$$
\begin{multline*}
\int_{\partial B_1} \left [\sigma (A_{x_j,R}[\nabla_j^\infty],P_{x_j,R}[\nabla_j^\infty]) n \right ]\cdot v  = \\  -\pi R^3\left (8  \asym( \nabla_1^\infty)\cdot \Omega + \frac{20}{3} \sym(\nabla_1^\infty) : E \right )\, \delta_{1j},
\end{multline*}
where $\delta_{1j}$ is the Kronecker symbol. \\
For the second term on the right hand side of formula \eqref{egalite2}, we consider $v_1:= U_{x_1,R}[V]+A_{x_1,R}[D]$ and write 
\begin{multline}\label{egalite3}
\int_{\mathbb{R}^3 \setminus \bigcup_i B_i} D U[u_*^{k}] : D(v) =\\ \int_{\mathbb{R}^3 \setminus \bigcup_i B_i} D( U[u_*^{k}]) : D(v_1) + \int_{\mathbb{R}^3 \setminus \bigcup_i B_i} D( U[u_*^{k}] ): D (v-v_1).
\end{multline}
To bound the last term we apply Lemma \ref{approximation_result} and Proposition \ref{convergence_reflection}  
\begin{equation*}
\begin{split}
&\underset{k \to \infty}{{\lim}} \left | \int_{\mathbb{R}^3 \setminus \bigcup_i B_i} D( U[u_*^{k}]) : D(v-v_1) \right |\\
&\lesssim  \underset{i}{\max}(|V_i|+ R[\Omega_i|)R  \left ( \frac{R}{\sqrt{d_{\min}^N}}|V| + \frac{R^3}{|d_{\min}^N|^{3/2}}|D| \right )\,,\\
& \lesssim \frac{R^2}{\sqrt{d_{\min}^N}} (|V|+R|D| )\,\underset{i}{\max}(|V_i|+ R|\Omega_i|).
\end{split}
\end{equation*}
We focus now on the first term on the right hand side of formula \eqref{egalite3}, we have
\begin{align*}
\left | \int_{\mathbb{R}^3 \setminus \underset{i}{\cup} B_i} D U[u_*^{(k)}] : D( v_1) \right | &= \left | \underset{i}{\sum} \int_{\partial B_i}[ \sigma(v_1,p_1) \cdot n] \cdot u_*^{(k)} \right |, \\
& \leq \underset{i}{\sum}4\pi R^2 \|\sigma(v_1,p_1)\|_{L^\infty( B_i)} \|u_*^{(k)}\|_{L^\infty( B_i)} ,
\end{align*}
using the decay properties \eqref{Stokeslet_maj}, \eqref{decay_rate} we have
\begin{eqnarray*}
\|\sigma(v_1,p_1)\|_{L^\infty(B_i)} & \lesssim & \frac{R|V|}{d_{i1}^2} + \frac{R^3}{d_{i1}^3}|D|, \text{ for $i \neq 1$}\,, \\ 
\|\sigma(v_1,p_1)\|_{L^\infty(B_1)}& \lesssim & \frac{|V|}{R} + |D|,
\end{eqnarray*}
hence
\begin{align*}
\left | \int_{\mathbb{R}^3 \setminus \underset{i}{\cup} B_i} D (U[u_*^{(k)}]) : D (v_1) \right | & \lesssim R (|V|+ R |D|) \|u_*^{(k)}\|_{L^\infty(B_1)}\,, \\
&+ R \underset{i\neq 1 }{\sum}\left (\frac{R^2 |V|}{d_{i1}^2}+ \frac{R^4|D|}{d_{i1}^3}  \right )\underset{i}{\max} \|u_*^{(k)}\|_{L^\infty( B_i)}\,,\\
& \lesssim   R (|V|+ R|D|)\, \underset{i}{\max} \|u_*^{(k)}\|_{L^\infty( B_i)}.
\end{align*}
According to Remark \ref{u_*_infty} , we have for all $1\leq i \leq N$
$$\underset{k\to \infty}{\overline{\lim}} \|u_*^{(k)}\|_{L^\infty(B_i)} \lesssim   R\, \underset{i}{\max}(|V_i|+ R|\Omega_i|). $$
Finally we get 
\begin{multline}\label{egalite4}
\underset{k \to \infty}{{\lim}} \left | \int_{\mathbb{R}^3 \setminus \bigcup_i B_i} D( U[u_*^{k}]) : \nabla v_1 \right |+ \left| \int_{\mathbb{R}^3 \setminus \bigcup_i B_i} D (U[u_*^{k}]) : \nabla (v-v_1)\right|\lesssim \\  \underset{i}{\max}(|V_i|+ R|\Omega_i|) \frac{R^2}{\sqrt{d_{\min}^N}}(|V| + R|D|).
\end{multline}
Identifying formula \eqref{egalite1} and \eqref{egalite2} and gathering all the inequalities above we have for all $V\,,\, \Omega \in \mathbb{R}^3$ and symmetric trace-free $3\times 3$ matrix $E$
\begin{multline*}
-V\cdot F_1 - \Omega \cdot T_1 - E : S_1 =
6 \pi R V_1^\infty \cdot V + 8 \pi R^3 \asym(\nabla_1^\infty)\cdot \Omega +  \frac{20}{3}\pi R^3 \sym(\nabla_1^\infty): E \\+ O \left (\underset{i}{\max}(|V_i|+ R|\Omega_i|) \frac{R^2}{\sqrt{d_{\min}^N}} (|V| + R|D|) \right),
\end{multline*}
with $ F_1 + m g=0$, $ T_1 =0$. Note that the value of the stresslet $S_i$, see \eqref{TS} for the definition, is unknown. However, we only need to approximate its value using Proposition \ref{strain_estimate}.
We conclude by identifying the terms involving $V\in \mathbb{R}^3$ to obtain 
$$
V_i^\infty := \underset{p=0}{\overset{\infty}{\sum}} V_i^{(p)} = \frac{m}{6 \pi R} g + O\left(  \underset{i}{\max}(|V_i|+ R|\Omega_i|) \frac{R}{\sqrt{d_{\min}^N}} \right),$$
for the skew-symmetric part we get
\begin{equation*}
R\left|\asym(\nabla_1^\infty) \right| \lesssim \underset{i}{\max}(|V_i|+ R|\Omega_i|) \frac{R}{\sqrt{d_{\min}^N}} \lesssim  \underset{i}{\max}(|V_i|+ R|\Omega_i|) \frac{R}{|d_{\min}^N|^{3/2}} ,
\end{equation*}
and for the symmetric part using Proposition \ref{strain_estimate}
\begin{equation*}
R\left|\sym(\nabla_1^\infty) \right| = O\left( \underset{i}{\max}(|V_i|+ R|\Omega_i|)\frac{R}{|d_{\min}^N|^{3/2}}\right)\,,
\end{equation*}
which concludes the proof.
\end{proof}
\begin{corollary}\label{corol}
Under the same assumptions as Lemma \ref{asymptotic_velocity}, there exists a positive constant $C=C(\kappa |g|)$ and $N(r_0,\bar{M},\mathcal{E})\in \mathbb{N}^*$ such that for all $N\geq N(r_0,\bar{M},\mathcal{E})$ we have 
$$\underset{1\leq i \leq N}{\max}\left( \left|V_i\right|+ R\left |\Omega_i \right|\right) \leq C.
$$
\end{corollary}
\begin{proof}
recall that $V_i^{(0)} =V_i$, $\nabla_i^{(0)} = \Omega_i$ for all $1\leq i \leq N$, according to Lemma \ref{asymptotic_velocity} and Lemma \ref {cvg_serie} we obtain  for all $1\leq i \leq N$
\begin{align*}
|V_i|+R |\Omega_i|& \leq  | V_i^\infty |+R |\nabla_i^\infty | +\underset{p=1}{\overset{\infty}{\sum}}\left( \left|V_i^{(p)}\right|+ R \left|\nabla_i^{(p)} \right|\right) \,,\\
& \leq |V_i^\infty|+R |\nabla_i^\infty | + K \left( \underset{p=0}{\overset{\infty}{\sum}} K^p \right) \underset{i}{\max}(|V_i|+ R|\Omega_i|) \,, \\
& \lesssim \kappa |g|+ \left( \frac{R}{|d_{\min}^N|^{3/2}}+ \frac{K}{1-K} \right) \underset{i}{\max}(|V_i|+ R|\Omega_i|).
\end{align*}
Hence, according to Lemma \ref{cvg_serie} we have $\frac{K}{1-K} <1$. Moreover, assumption \eqref{hyp4} ensures that
 $$\frac{R}{|d_{\min}^N|^{3/2}}\lesssim  \frac{\mathcal{E}^{3/4} \bar{M}^{3/4}}{N^{1/4}}\,,$$  
which vanishes when $N$ goes to infinity.
\end{proof}
\subsection{Extraction of the first order terms for the velocities $(V_i,\Omega_i)$} 
In order to control the motion of the particles, we want to provide a good approximation of the unknown velocities $(V_i,\Omega_i)$. Thanks to the method of reflections, the velocity field $u^N$ can be approached by a superposition of analytical solutions to a Stokes flow generated by a translating, a rotating and a straining sphere (See Proposition \ref{convergence_reflection}) with the associated velocities $(V_i^\infty, \nabla^\infty_i)$. This allows us to compute the first order terms for $(V_i,\Omega_i)$ applying Lemma \ref{asymptotic_velocity} and Corollary \ref{corol}. Keeping in mind that all the computations are done for a fixed time $t \geq 0$, the main result of this Section is the following Proposition.
\begin{proposition}\label{velocity_formula}
Assume that, for a fixed time, we have the existence of a sequence $(\lambda^N)_{N \in \mathbb{N}^*}$ and two positive constants $\bar{M}, \mathcal{E}$ such that $(X^N)_{N \in \mathbb{N}^*} \in \mathcal{X}(\bar{M}, \mathcal{E})$. Assume moreover that $\bar{M}^{1/3 }r_0 $ is small enough. Then, there exists $N(r_0,\bar{M},\mathcal{E}) \in \mathbb{N}^*$ such that for all $N\geq N(r_0,\bar{M},\mathcal{E})$, for all $1 \leq i \leq N$ we have 
$$
V_i= \kappa g +   6 \pi R \underset{j\neq i}{\overset{N}{\sum}} \Phi(x_i-x_j)\kappa g+ O\left (d_{\min}^N \right ), \:\: R \Omega_i =  O\left ( d_{\min}^N \right ),
$$
\end{proposition}
We begin by the following lemma: 
\begin{lemma}\label{Oseen_Stokeslet_maj}
For all trace-free $3 \times 3$ matrices $(D_i)_{1\leq i \leq N}$, for all $W \in \mathbb{R}^{3}$ and  $1\leq i \leq N$ we have
$$
\underset{j\neq i}{\overset{N}{\sum}}\left| 6 \pi R \Phi(x_i-x_j)\, W - U_{x_j,R}[W](x_i) \right| \lesssim R |W|.
$$
$$
 \underset{j\neq i}{\overset{N}{\sum}}\left| A_{x_j,R}[D_j](x_i) \right | \lesssim R  \, \underset{j}{\max}\, R|D_j|.
$$
\end{lemma}
\begin{proof}[Proof of Lemma \ref{Oseen_Stokeslet_maj}]
Thanks to formula \eqref{Oseen_Stokeslet} we have for $i\neq j$
\begin{align*}
U_{x_j,R}[W](x_i)& =  6 \pi R\Phi(x_j-x_i)W + \frac{1}{4} \frac{R^3}{|x_j-x_i|^3} W - \frac{3}{4} R^3 \frac{(x_j-x_i)\cdot W}{|x_j-x_i|^5} (x_j-x_i),
\end{align*}
this yields 
$$
\left |U_{x_j,R}[W](x_i)- 6 \pi R\Phi(x_j-x_i)W  \right | \lesssim  \frac{R^3}{d_{ij}^3}|W|.
$$
Applying Lemma \ref{JO} with $k=3$ yields
\begin{multline*}
\underset{j\neq i}{\overset{N}{\sum}} \left |U_{x_j,R}[W](x_i)- 6 \pi R\Phi(x_j-x_i)W  \right | \lesssim \underset{j\neq i}{\overset{N}{\sum}}\frac{R^3}{d_{ij}^3}|W| \\ \lesssim R r_0\bar{M} \left (\frac{R |\lambda^N|^3}{|d_{\min}^N|^3}+ R (| \log \bar{M}|+ | \log \lambda^N |)\right )  |W|.
\end{multline*}
We have thanks to assumptions \eqref{hyp1}, \eqref{hyp3} and \eqref{Rdmin}
$$
\frac{R |\lambda^N|^3}{|d_{\min}^N|^3}+ R (| \log \bar{M}|+ | \log \lambda^N |) \leq \frac{R}{d_{\min}^N} \mathcal{E} + R| \log \bar{M}| + R \log N \lesssim 1.
$$
Analogously, we obtain the second bound by applying \ref{JO} with $k=2$ this time.
\end{proof}
We can now prove the main result.
\begin{proof}[Proof of Proposition \ref{velocity_formula}]
Let fix $1\leq i \leq N$. According to Lemma \ref{asymptotic_velocity} and Corollary \ref{corol} we have
$$
V_i^\infty = \underset{p=0}{\overset{\infty}{\sum}} V_i^{(p)} = \frac{m}{6 \pi R} g + O\left(\frac{R}{\sqrt{d_{\min}^N}} \right).$$
As $V_i^{(0)}=V_i$ we get 
\begin{align*}
V_i= - \underset{p=1}{\overset{\infty}{\sum}} V_i^{(p)} + \frac{m}{6 \pi R} g + O\left ( \frac{R}{\sqrt{d_{\min}^N}} \right).
\end{align*}
Formula \eqref{formula_reflection} for the velocities $V_j^{(p)}$ yields 
\begin{align*}
V_i& =  \underset{p=1}{\overset{\infty}{\sum}} \underset{j\neq i}{\sum} \left( U_{x_j,R}[V_j^{(p-1)}](x_i)+ A_{x_j,R}[\nabla_j^{(p-1)}](x_i) \right) + \frac{m}{6 \pi R} g + O\left( \frac{R}{\sqrt{d_{\min}^N}} \right)\,,\\
& =  \frac{m}{6 \pi R} g+ \underset{j\neq i}{\sum} \left ( U_{x_j,R}[V_j^{\infty}](x_i)+ A_{x_j,R}[\nabla_j^{\infty}](x_i)  \right) + O\left(\frac{R}{\sqrt{d_{\min}^N}} \right),
\end{align*}
we apply Lemma \ref{Oseen_Stokeslet_maj}, Lemma \ref{asymptotic_velocity} and Corollary \ref{corol} together with \eqref{hyp4} and \eqref{Rdmin} to get: 
\begin{align*}
\underset{j\neq i}{\sum} \left |A_{x_j,R}[\nabla_j^{\infty}](x_i)\right|& \lesssim R \, \underset{j}{\max} R |\nabla_j^\infty|\,, \\
& \lesssim  R \frac{R}{|d_{\min}^N|^{3/2}}\,, \\
& \leq d_{\min}^N \frac{\mathcal{E}^{3/4} \bar{M}^{3/4}}{N^{1/4}}\,,\\
&\lesssim  d_{\min}^N .
\end{align*}
Now, we rewrite the sum as follows: 
\begin{align*}
 \underset{j\neq i}{\sum} U_{x_j,R}[V_j^{\infty}](x_i)& =  \underset{j\neq i}{\sum}  U_{x_j,R}[\kappa g](x_i)+  \underset{j\neq i}{\sum}  U_{x_j,R} \left [V_j^{\infty}-\kappa g \right](x_i),
\end{align*}
and we bound the error term using the decay rate \eqref{Stokeslet_maj}, Lemma \ref{asymptotic_velocity} and Lemma \ref{JO} with $k=1$  
\begin{align*}
\left |\underset{j\neq i}{\sum}  U_{x_j,R} \left [V_j^{\infty}-\kappa g \right](x_i) \right| & \lesssim  \left(\underset{j\neq i}{\sum} \frac{R}{d_{ij}} \right) \underset{j}{\max} \left| V_j^{\infty}-\kappa g\right|\,,\\
&\lesssim  \frac{R}{\sqrt{d_{\min}^N}}\,,\\
&\lesssim  d_{\min}^N.
\end{align*}
We conclude by replacing the stokeslets by the Oseen tensor thanks to Lemma \ref{Oseen_Stokeslet_maj}. Finally we have for all $1 \leq i \leq N$
$$
V_i= \kappa g +   6 \pi R \underset{j\neq i}{\overset{N}{\sum}} \Phi(x_i-x_j)\kappa g+ O\left ( d_{\min}^N\right ).
$$
For the angular velocities we obtain thanks to Lemma \ref{asymptotic_velocity} and formula \eqref{formule_nabla} for $\nabla_1^{(p)}$, $ p\geq 1$
\begin{align*}
R\Omega_1 &= - \underset{p=1}{\overset{\infty}{\sum}} R\asym{\nabla_1^{(p)}} + O\left (\frac{R}{\sqrt{d_{\min}^N}}\right )\,,\\
& = R\asym \left (  \underset{j \neq 1}{\sum} \nabla U_{x_j,R}[V_j^\infty](x_1) + \nabla A_{x_j,R}[ \nabla_j^\infty](x_1) \right )+O\left ( \frac{R}{\sqrt{d_{\min}^N}}\right ).
\end{align*}
As before, using Lemma \ref{asymptotic_velocity} we bound the first term by 
\begin{align*}
\phantom{}& R \left |\underset{j \neq 1}{\sum} \nabla U_{x_j,R}[V_j^\infty](x_1) + \nabla A_{x_j,R}[ \nabla_j^\infty](x_1) \right | \\
& \lesssim R \left(\underset{j\neq 1}{\sum} \frac{R}{d_{1j}^2} + \frac{R^2}{d_{1j}^3} \right ) \underset{j}{\max} (|V_j^\infty|, R | \nabla_j^\infty|) \,, \\
& \lesssim R \left(\underset{j\neq 1}{\sum} \frac{R}{d_{1j}^2} \right )\left (1+\frac{R}{d_{\min}^N} \right) \underset{j}{\max} (|V_j^\infty|, R | \nabla_j^\infty|)\,, \\
&\lesssim Rr_0 \left ( \frac{|\lambda^N|^3}{|d_{\min}^N|^2}\bar{M}+\bar{M}^{2/3} \right)\,,\\
&\lesssim  d_{\min}^N\left(\frac{R|\lambda^N|^3}{|d_{\min}^N|^3}+\bar{M}^{2/3} \right)\,,\\
& \lesssim d_{\min}^N \,,
\end{align*}
where we used the fact that $\frac{R |\lambda^N|^3}{|d_{\min}^N|^3}$ is uniformly bounded according to \eqref{hyp1} and \eqref{Rdmin}.
\end{proof}
\section{Control of the particle distance and concentration}\label{control_particles}
In this Section, we make precise the particle behavior in time. Precisely we want to prove that if initially there exists two positive constants $\bar{M}, \mathcal{E}$ and a sequence $(\lambda^N)_{N \in \mathbb{N}^*}$ such that $(X^N(0))_{N \in \mathbb{N}^*} \in \mathcal{X}(\bar{M}, \mathcal{E})$ (see Definition \ref{def_regime}), then the same holds true for a finite time.
Recall that the initial distribution of particles satisfies: 
\begin{itemize}
\item The minimal distance is at least of order $|\lambda^N|^{3/2}$.
\item The maximal number of particles concentrated in a cube of width $\lambda^N$ satisfies assumption \eqref{bound_concentration}.
\end{itemize}
We aim to show that there exists a small interval of time $[0,T]$ independent of $N$ on which the particle distance and concentration stay at the same order. The idea is to use a Gronwall argument and the computation of the velocities $(V_i)_{1 \leq i \leq N}$ at each fixed time $t\geq 0$.
\subsection{Proof of Theorem \ref{thm1}}
We assume that initially there exists two positive constants $\bar{M}, \mathcal{E}$ and a sequence $(\lambda^N)_{N \in \mathbb{N}^*}$ such that $(X^N(0))_{N \in \mathbb{N}^*} \in \mathcal{X}(\bar{M}, \mathcal{E})$. Let $T>0$ be such that
\begin{equation}\label{continuite}
d_{ij}(t) \geq \frac{1}{2} d_{ij}(0) \:, \forall 1 \leq i \neq j \leq N\,,\, \forall t \in[0,T[.
\end{equation}
This maximal time $T>0$ exists and we aim to prove that it is independent of $N$. As long as $t < T$ we have a control on the particle concentration.
\begin{lemma}[Control of particle concentration $M^N$]\label{lemma2}\label{lemma_control}
As long as $t \in[0,T[$ we have:
$$
M^N(t) \leq 8^4 M^N(0).
$$
\end{lemma}
\begin{proof}
We recall the definition of $M^N(t)$
$$
M^N(t) := \underset{x \in \mathbb{R}^3}{\sup} \Big\{ \# \big\{ i \in \{1,\cdots,N\} \text{ such that } x_i(t) \in \overline{B_\infty(x,\lambda^N)}  \big \}\Big\}.
$$
We introduce the following quantity: 
\begin{equation}\label{concentration_bis}
L^N(t):= \underset{i}{\max} \# \left \{j \in \{1,\dots,N\}  \text{ such that } |x_i(t)-x_j(t)|_\infty \leq \lambda^N) \right \}.
\end{equation}
One can show that the two definitions of concentration $L^N(t)$ and $M^N(t)$ are equivalent in the sense that
$$
L^N(t) \leq  M^N(t) \leq 8 L^N(t)
$$
see Lemma \ref{equivalence} for the proof. We also need to introduce the following notation for all $\beta >0$: 
$$
L^N_\beta(t):= \underset{i}{\max} \# \left \{j \in \{1,\dots,N\} \text{ such that } |x_i(t)-x_j(t)|_\infty \leq \beta \lambda^N) \right \},
$$
and
$$
M^N_\beta(t) := \underset{x}{\sup} \Big\{ \# \big\{ i \in \{1,\cdots,N\} \text{ such that } x_i(t) \in \overline{B_\infty(x,{\beta}\lambda^N))}  \big \}\Big\},
$$
with the notation 
\begin{eqnarray*}
M^N_1(t):= M^N(t)\,,&L^N_1(t):= L^N(t).
\end{eqnarray*}
We have for all $\beta>0$ and all $\alpha>1$
$$
L^N_{\alpha \beta} (t) \leq 8 \lceil \alpha \rceil^3 L_\beta^N(t),
$$
where $\lceil \cdot \rceil$ denotes the ceiling function. See Corollary \ref{scaling_concentration} for the proof.

The idea is to show that the concentration $L^N$ is controlled in time and hence, the same applies to $M^N$ according to Lemma \ref{equivalence}. 
Recall that we have for all $t \in[0,T[$
$$
d_{ij}(t) \geq \frac{1}{2} d_{ij}(0).
$$
Now, fix $1 \leq i \leq N$ and consider $j\neq i$ satisfying $|x_i(0)-x_j(0)|_\infty> \lambda^N$, then 
\begin{align*}
|x_i(t)-x_j(t)|_\infty &\geq \frac{1}{\sqrt{3}}|x_i(t)-x_j(t)|\,,\\
& \geq\frac{1}{2\sqrt{3}}|x_i(t)-x_j(0)|\,,\\
& > \frac{\lambda^N}{2\sqrt{3}}.
\end{align*}
Which means that $$j \not\in \left \{ 1 \leq k \leq \, N, \text{ such that } |x_i(t)-x_k(t)| \leq \frac{\lambda^N}{2\sqrt{3}} \right \}.$$ 
We obtain 
\begin{multline} 
\left \{ 1 \leq j \leq N\, , \text{ such that } |x_i(t) -x_j(t)| \leq \frac{\lambda^N}{2\sqrt{3}} \right \}
\\
\subset\left  \{ 1 \leq j \leq N\, , \text{ such that } |x_i(0) -x_j(0)| \leq \lambda^N \right \}.
\end{multline}
Hence taking the maximum over $1\leq i \leq N$ we obtain 
$$
L^N_{\frac{1}{2\sqrt{3}}}(t) \leq L^N(0),
$$
thus, we apply Corollary \ref{scaling_concentration} with $\beta = \frac{1}{2\sqrt{3}}$ and $\alpha=\beta^{-1}=2\sqrt{3}$ to get 
\begin{align*}
L^N(t) \leq 8^3 L^N(0).
\end{align*}
According to Lemma \ref{equivalence}, the equivalence between $M^N$ and $L^N$ yields finally for all $t \in[0,T[$ 
$$
M^N(t) \leq 8^4 M^N(0).
$$
This ends the proof.
\end{proof}
This shows that as long as $t <T$ we have $(X^N(t))_{N \in \mathbb{N}^*} \in \mathcal{X}(8^4 \bar{M}, 4 \mathcal{E})$. This implies the following control.
\begin{proposition}\label{prop1}
Assume that there exists two positive constants $\bar{M}, \mathcal{E}$ and a sequence $(\lambda^N)_{N \in \mathbb{N}^*}$ such that $(X^N)_{N \in \mathbb{N}^*} \in \mathcal{X}(8^4\bar{M}, 4\mathcal{E})$. If $r_0 \bar{M}^{1/3}$ is small enough, there exists $N(r_0,\bar{M},\mathcal{E})$ and a positive constant $C=C(r_0, \bar{M}, \mathcal{E}, \kappa |g|)$ independent of $N$ such that for all $ N \geq N(r_0,\bar{M},\mathcal{E})$, for all $i \neq j$ we have
$$
|V_i-V_j| \leq C d_{ij}.
$$
\end{proposition}
\begin{proof}
For the sake of clarity we fix $i=1$ and $j=2$. The computations below are independent of this choice. Thanks to Proposition \ref{velocity_formula} we obtain :
$$
V_1-V_2 = 6\pi R \underset{i\neq 1,2}{\overset{N}{\sum}} ( \Phi(x_1-x_i) - \Phi(x_2-x_i)) \kappa g + O( d_{\min}^N).
$$
Hence, according to assumption \eqref{hyp1}, formula \eqref{lipschitz} and using Lemma \ref{JO} for $k=2$ we obtain  
\begin{align*}
|V_1-V_2| &\lesssim R \underset{i\neq 1,2}{\overset{N}{\sum}} \left (\frac{1}{d_{1i}^2} + \frac{1}{d_{2i}^2} \right )|x_1-x_2| +O( d_{\min}^N)\,,\\
& \lesssim r_0 \left (\bar{M}\frac{|\lambda^N|^3}{|d_{\min}^N|^2}+\bar{M}^{2/3} \right) |x_1-x_2| + O( d_{\min}^N)\,,\\
&\lesssim  d_{12}.
\end{align*}
We set then $C>0$ the universal constant implicit in the above estimate. 
\end{proof}
We have the following control.
\begin{lemma}[Control of particle distance]\label{lemma1}
For all $1 \leq i\neq j \leq N$, for all $t\in[0,T[$ we have
$$
d_{ij}(t) \geq d_{ij}(0)e^{-Ct}.
$$
\end{lemma}
\begin{proof}
Thanks to \eqref{continuite} and Lemma \ref{lemma2} we have for all $t<T$ that 
$$(X^N(t))_{N\in\mathbb{N}^*} \in \mathcal{X}(8^4 \bar{M}, 4\mathcal{E}).$$
Hence, all computations from Proposition \ref{prop1} hold true up to time $T$. In other words, there exists a positive constant $C=C(r_0, \bar{M}, \mathcal{E},\kappa |g|)$ such that for all indices $1 \leq i\neq j \leq N$ we have 
$$
|V_i(t)-V_j(t)| \leq C \, d_{ij}(t) \, \forall t \in[0,T[,
$$
thus,
\begin{align*}
\frac{d}{dt} d_{ij}(t) &\geq - |V_i(t)-V_j(t)|, \\
& \geq - C \, d_{ij}(t).
\end{align*}
This entails
$$
d_{ij}(t) \geq d_{ij}(0)e^{-Ct}\,,
$$
which is the desired result.
\end{proof}
\subsection*{Conclusion}
Thanks to Lemma \ref{lemma1} and Lemma \ref{lemma_control} we have for all $ 1 \leq i\neq j \leq N$, $t\in[0,T[$ 
\begin{eqnarray*}
d_{ij}(t)&\geq& d_{ij}(0) e^{-Ct}, \\
M^N(t) &\leq & 8^4 M^N(0),
\end{eqnarray*}
this shows that $T$ is independent of $N$ and is at least of order  $\frac{\log(2)}{C}$ where $C$ depends on $(r_0,\bar{M},\mathcal{E},\kappa|g|)$.
\section{Reminder on Wasserstein distance and analysis of the limiting equation}
In this part we recall some important results of existence, uniqueness, regularity and stability concerning the mean-field equation \eqref{mean_field}. We recall also the definition of the Monge-Kantorovich-Wasserstein distance of order one and infinite. We refer to \cite[Part I, chapter 6]{Villani} for definition and properties of the order one distance $W_1$. To define the infinite Wasserstein distance we start with some associated notions. We refer to \cite{Champion} for more details.
\begin{definition}[Transference plan]
Let $\mu\,,\, \nu \in \mathcal{P}(\mathbb{R}^3)$ be two probability measures. The set of transference plans from $\mu$ to $\nu$ denoted $\Pi (\mu\,,\, \nu )$ is the set of all probability measures $\pi \in \mathcal{P}(\mathbb{R}^3 \times \mathbb{R}^3)$ with first marginal $\mu$ and second marginal $\nu$ \textit{i.e.} 
$$
\pi \in \Pi (\mu\,,\, \nu )\Leftrightarrow \int \int_{\mathbb{R}^3 \times \mathbb{R}^3} ( \phi(x)+\psi(y)) \pi(dxdy) = \int_{\mathbb{R}^3} \phi(x) \mu(dx) + \int_{\mathbb{R}^3} \psi(y) \nu(dy),
$$
for all $\phi\,,\, \psi \in \mathcal{C}_b(\mathbb{R}^3)$.
\end{definition}
Recall that for all probability measure  $\lambda \in \mathcal{P}(\mathbb{R}^3 \times \mathbb{R}^3)$ we have
\begin{definition}[Essential supremum]
$$
\lambda - \esssup |x-y| := \inf \{ t \geq 0 \,:\, \lambda (\{(x,y)\in  \mathbb{R}^3 \times \mathbb{R}^3 \,:\, |x-y|> t  \})=0 \}.
$$
\end{definition}
We recall also the definition of the support for a (non-negative) measure.
\begin{definition}[Measure support]
Given $\mu \in \mathcal{P}(\mathbb{R}^3)$ a non-negative measure, then the support of $\mu$ is defined as the set of all points $x$ for which every open neighbourhood of $x$ has positive measure
$$
\supp \mu =\{ x\in \mathbb{R}^3 : \forall \,V \in  \mathcal{V}(x) \,,\, \mu (V) >0 \},
$$
where $\mathcal{V}(x)$ denotes the set of open neighbourhoods of $x$.
\end{definition}
With this definition for the support one can show that there holds
$$
\lambda -  \esssup |x-y| := \sup \{|x-y| \,:\, (x,y) \in \text{ supp } \lambda   \}).
$$
We can now define the infinite Wasserstein distance $W_\infty$:
\begin{definition}[Infinite Wasserstein distance]\label{def_Wasserstein}
The infinite Wasserstein distance between two probability measures $\mu$ and $\nu$ is defined as follows:
$$
W_\infty (\mu\,,\,\nu) = \underset{\pi \in \Pi(\mu\,,\,\nu)}{\inf} \{\pi -  \esssup |x-y| \} .
$$
A transference plan $\pi^* \in \Pi(\mu,\nu)$  satisfying
$$
W_\infty (\mu\,,\,\nu)=\pi^* -  \esssup |x-y|,
$$
 is called an optimal transference plan.
\end{definition}
We recall also the definition of a transport map.
\begin{definition}[Transport map]
Given two probability measures $\mu$ and $\nu$, a transport map $T$  is a measurable mapping  $T : \supp \mu \to \mathbb{R}^3 $ such that  $$ \nu = T_{\#} \mu. $$
\end{definition} 
We emphasize that $T(\mathbb{R}^3) \subset \supp \nu$ $\mu$ - almost everywhere. Indeed 
\begin{align*}
\mu \{ x \in \mathbb{R}^3 : T(x) \notin \text {supp $\nu$ }  \} & = \mu \{ T^{-1}( {}^c \text {supp $\nu$ }) \} \,,\\
& = \nu \{  {}^c \supp \nu\}\,, \\
&=0.
\end{align*}
\begin{remark}
Note that, for all transport map $T$ from $\mu$ to $\nu$ one may associate a transference plan $(Id,T)\# \mu \in \Pi(\mu\,,\,\nu)$ \textit{i.e.} the pushforward of $\mu$ by the map $x\mapsto (x,T(x))$ and we have 
\begin{align*}
&(Id,T)\# \mu -  \esssup |x-y|\,, \\
&= \inf \{ t \geq 0 \,:\, (Id,T)\# \mu (\{(x,y)\in  \mathbb{R}^3 \times \mathbb{R}^3 \,:\, |x-y|\geq t  \})=0 \}\,,\\
& = \inf \{ t \geq 0 \,:\, \mu ( (Id,T)^{-1}\{(x,y)\in  \mathbb{R}^3 \times \mathbb{R}^3 \,:\, |x-y|\geq t  \})=0 \}\,,\\
& = \inf \{ t \geq 0 \,:\, \mu (\{x\in  \mathbb{R}^3 \,:\, |x-T(x)|\geq t  \})=0 \}\,,\\
&=\mu -  \esssup |x-T(x)|.
\end{align*}
\end{remark}
Note that this yields 
$$
\underset{\pi \in \Pi(\mu\,,\,\nu)}{\inf} \{\pi -  \esssup |x-y| \}  \leq {\inf} \{\mu -  \esssup |T(x)-x|\:,\: T : \supp \mu \to \mathbb{R}^3 \,,\, \nu =T\# \mu \}.
$$
It is then natural to investigate in which conditions one has the existence of a transport map $T$ associated to an optimal transference plan.
 As in \cite{HJ} we refer to \cite{Champion} for the following existence result.
\begin{theorem}[Champion, De Pascale, Juutinen]\label{existence_transport_map}
Assume that $\mu$ is absolutely continuous with respect to the Lebesgue measure. Then there exists optimal transference plans, and at least one of them is given by a transport map $T$. If moreover $\nu$ is a finite sum of Dirac masses, this optimal transport map is unique.
\end{theorem}
\subsection{Existence, uniqueness and stability for the mean-field equation}\label{Vlasov_existence}
Consider the following problem 
\begin{equation}\label{Vlasov_Stokes}
\left\{
\begin{array}{rcl}
\frac{\partial \rho}{\partial t }+{\div}( ( \kappa g + \mathcal{K}\rho) \rho ) &=& 0\,, \, \\
\rho(0,\cdot) &= & \rho_0 \,,
\end{array}
\right.
\end{equation}
recall the definition of the kernel $\mathcal{K}$  
$$
\mathcal{K} \eta (x) = {6 \pi r_0 \kappa }\int \Phi(x-y)\,g\, \eta(y) dy ,
$$
for all $\eta\in L^\infty (\mathbb{R}^3) \times L^1(\mathbb{R}^3)$.
We refer to the existence and uniqueness result due to H\"ofer \cite[Theorem 9.2]{Hofer} in the case where the initial data $\rho_0$ and its gradient $\nabla \rho_0$ are in the space $X_\beta$ for some $\beta >2$ where
$$
X_\beta := \{ h \in L^\infty(\mathbb{R}^3) \,,\, \|h\|_{X_\beta} < \infty \},
$$
with 
$$
\|h\|_{X\beta} := \underset{x}{\esssup}(1+|x|^\beta)|h(x)|.
$$
\begin{theorem}[H\"ofer]\label{thm_Vlasov}
Assume that $\rho_0$, $\nabla \rho_0 \in X_\beta$ for $\beta>2$. There exists a unique solution $\rho\in W^{1,\infty}((0,T), X_\beta)$ to equation \eqref{Vlasov_Stokes} for all $T>0$ and a unique well defined flow $X$ satisfying 
\begin{equation}\label{flow_X}
\left \{
\begin{array}{rcll}
\partial_s X(s,t,x) & = & \kappa g + \mathcal{K}\rho(s, X(s,t,x)),&\: \: \forall \, s,t \in [0,+\infty[ , \\
X(t,t,x) & = & x,& \: \: \forall \, t \in [0,+\infty[ ,
\end{array}
\right.
\end{equation}
such that 
\begin{equation}\label{formula_rho}
\rho(t,x)= \rho_0(X(0,t,x)) \,,\:\: \forall \, (t,x) \in[0,+\infty[ \times \mathbb{R}^3.
\end{equation}
\end{theorem}
\begin{remark}\label{rem2}
The flow $X$ is measure-preserving \textit{i.e.} for a test function $\phi \in \mathcal{C}_b(\mathbb{R}^3)$ we have
$$
\int \phi(y) {\rho}(s,y)dy = \int \phi(X(s,t,y)) \rho(t,y)dy,
$$
for all $s\,,\,t \in[0,T]$. This allows us to separate the dependence of time $s$ in the integral with respect to the measure ${\rho}(t,\cdot)$.
\end{remark}
\begin{remark}\label{remarque2}
Note that for all $ \eta \in L^\infty(\mathbb{R}^3) \cap L^1(\mathbb{R}^3)$,  the velocity field $\mathcal{K} \eta$ is Lipschitz
$$
\left|\mathcal{K}(\eta)(x) - \mathcal{K}(\eta)(y) \right| \lesssim ( \|\eta\|_{L^1}+\|\eta\|_{L^\infty})  \,|x-y|,\,\:\:\forall \, x \neq y \in \mathbb{R}^3.
$$
\end{remark}
Moreover, if one assume that $\rho_0$ is only Lipschitz and compactly supported, then one can show the existence and uniqueness of the solution $\rho$ to equation \eqref{Vlasov_Stokes} in the space $L^\infty((0,T); L^\infty(\mathbb{R}^3) \cap L^1(\mathbb{R}^3))$. The method of proof is related to the stability result due to G. Loeper in \cite{Loeper}  which gives a stability estimate in terms of Wasserstein distance for the Vlasov-Poisson equation. This result is adapted by M. Hauray in \cite[Theorem 3.1]{Hauray} for a more general class of kernels $K$ satisfying a $(C^\alpha)$ condition with $ \alpha<d-1$ where $d$ is the space dimension
\begin{equation}\tag{$C^\alpha$}
\div K =0,\,|K(x)|,\, |x| | \nabla K(x)| < \frac{C}{|x|^\alpha},\,\forall \, x \neq 0 , 
\end{equation}
see \cite{Hauray}. This condition being satisfied by the Oseen tensor $\Phi$ we have the following result.
\begin{theorem}[Hauray-Loeper]\label{stability_Loeper}
Given $T>0$, consider two solutions $\rho_1, \rho_2 \in L^\infty( (0,T), L^\infty(\mathbb{R}^3) \cap L^1(\mathbb{R}^3))$ of equation \eqref{Vlasov_Stokes} associated to two initial data  $\rho_0^1,\rho_0^2 \in L^\infty(\mathbb{R}^3) \cap L^1(\mathbb{R}^3)$. There holds
\begin{equation}
W_1(\rho_1(t,\cdot), \rho_2(t,\cdot)) \leq W_1(\rho_0^1, \rho_0^2)e^{C \max( \|\rho_1^0\|_{L^\infty\cap L^1},\|\rho_2^0\|_{L^\infty\cap L^1} )t}.
\end{equation}
\end{theorem}
We refer to \cite[Theorem 3.1]{Hauray} for a complete proof which introduces the main ideas used also in \cite{HJ} for the mean field approximation result.
\subsection{$\rho^N$ as a weak solution to a transport equation}
According to Theorem \ref{thm1}, there exists a time $T>0$ independent of $N$ for which the particles do not overlap. This shows that the empirical measure
$$
\rho^N(t,x):= \frac{1}{N} \underset{i=1}{\overset{N}{\sum}} \delta_{x_i(t)}(x),
$$
is well defined on $[0,T]$. Recall that we are interested in the limiting behaviour of $\rho^N\in \mathcal{P}([0,T] \times \mathbb{R}^3)$ when $N \to \infty$. According to Proposition \ref{velocity_formula}, particles $(x_i)_{1 \leq i \leq N}$ satisfy the following system: \begin{equation*}
\left \{
\begin{array}{rcl}
\dot{x}_i &= & V_i,\\
V_i & \sim & \kappa g + 6\pi R \underset{i\neq j}{\sum} \Phi(x_i-x_j).
\end{array}
\right.
\end{equation*}
In order to prove Theorem \ref{thm2} we want to compare the particle system to the continuous density $\rho$ which is solution to equation \eqref{Vlasov_Stokes}. Hence, we need to express $\rho^N$ as a weak solution to a transport equation. The remainder of this Section is devoted to establish such a formulation.\\
Analogously to the continuous case, we are interested in giving a sense to the quantity 
$$
\mathcal{K} \rho^N = 6\pi r_0 \kappa \int \Phi(x-y) g \rho^N(t,dy),
$$
which is not well defined because $\Phi$ is singular. On the other hand, as the only values of $\Phi$ that matters are the terms $\Phi(x_i-x_j)$, $i\neq j $ we define the following regularization 
$$
\psi^N\Phi(x):= \Phi(x)\psi^N(x), 
$$
where $\psi^N(x):= \psi \left (\frac{x}{d_{\min}^N(0)}\right )$ and $\psi$ is a truncation function such that $\psi =0$ on $B(0,1/4)$ and $\psi=1$ outside $B(0,1/2)$. We can now define the operator $\mathcal{K}^N$
\begin{align*}
\mathcal{K}^N \rho^N (t,x)& := 6 \pi r_0 \kappa \int_{\mathbb{R}^3} \psi^N \Phi(x-y)\,g\, \rho^N(t,dy)\,, \\
& = \frac{6\pi r_0 \kappa}{N}\underset{i}{\sum} \psi^N\Phi(x-x_i(t)) g.
\end{align*}
Since Theorem \ref{thm1} ensures that the particles satisfy
$$
|x_i(t)-x_j(t)| \geq \frac{1}{2} d_{\min}^N(0) \:,\: \: \forall i\neq j\,, \forall\, t\in[0,T],
$$
we have for $x=x_i(t)$, $t\in[0,T]$, $1 \leq i \leq N$
$$
\mathcal{K}^N \rho^N (t,x_i(t))=\frac{6\pi r_0 \kappa}{N}\underset{j\neq i}{\sum} \Phi(x_j(t)-x_i(t))g.
$$
Now, it remains to check that $\rho^N$ is a weak solution of a transport equation. We recall that $\rho^N$ is a weak solution of a transport equation $\frac{\partial }{\partial t }+ \div( V \rho^N) $ with $V\in \mathcal{C}([0,T],\mathcal{C}^1(\mathbb{R}^3))$ if for all test function $\phi \in \mathcal{C}^\infty_c ( [0,T] \times \mathbb{R}^3 )$ we have
$$
\int_0^T \int_{\mathbb{R}^3} \left ( \partial_t \phi(t,x) + \nabla \phi(t,x) \cdot V(t,x) \right ) \rho^N(dx,t) dt  =0.
$$
Note that this integral yields
\begin{align*}
 \phantom{=}&   \int_0^T \int_{\mathbb{R}^3} \left ( \partial_t \phi(t,x) + \nabla \phi(t,x) \cdot V(t,x) \right ) \rho^N(dx,t) dt\,,  \\ 
 &=    \int_0^T \frac{1}{N} \underset{i}{\sum} \left ( \partial_t \phi(t,x_i(t)) + \nabla \phi(t,x_i(t)) \cdot V(t,x_i(t)) \right ).
 \end{align*}
In particular if we choose $V$ such that $V(t,x_i(t))= V_i(t)$ one has
 \begin{align*}
 & =   \int_0^T \frac{1}{N} \underset{i}{\sum}  \partial_t \phi(t,x_i(t)) + \nabla \phi(t,x_i(t)) \cdot V_i \,, \\ 
&= \frac{1}{N} \underset{i}{\sum} \int_0^T \frac{d}{dt}( \phi(t,x_i(t)) ) \,, \\ 
 &=0.
 \end{align*}
On the other hand, we recall that from Proposition \ref{velocity_formula} we can write for all $1\leq i \leq N$ 
\begin{align*}
V_i&= \kappa g +   6 \pi R \underset{j\neq i}{\overset{N}{\sum}} \Phi(x_i-x_j)\kappa g+ E_i (t) \,,\\
& = \kappa g + \mathcal{K}^N \rho^N(t,x_i(t)) + E_i(t),
\end{align*}
with $E_i(t) = O(d_{\min}^N)$. Hence if we construct a divergence-free vector field $E^N$ such that 
$$
E^N(t,x_i(t))= E_i(t),
$$
we can define $V$ as 
$$
V(t,x)= \kappa g + \mathcal{K}^N \rho^N(t,x)+E^N(t,x).
$$
\paragraph{\textbf{Construction of $E^N$}}\label{E^N_construction}
We fix $\chi$ a truncation function such that $\chi=1 $ on $B(0,1)$ and $\chi = 0 $ on ${}^c B(0,2)$. For all $i$ we set
$$
\mathcal{E}_i(t,x):= \text{curl} \left ( \chi \left (\frac{x-x_i(t)}{R} \right ) E_i(t) \times \frac{x-x_i(t)}{2} \right ). 
$$ 
By construction, $\mathcal{E}_i$ is a divergence-free compactly supported vector field satisfying
$$
\mathcal{E}_i(t,x_i(t))= E_i(t).
$$
Furthermore, $\mathcal{E}_i$ is supported in $B(x_i(t),2R)$. Thanks to Theorem \ref{thm1}, this entails that $\supp (\mathcal{E}_i) \cap \supp (\mathcal{E}_j)= \emptyset$ for $i\neq j$. We set then 
$$
E^N(t,x):= \underset{i}{\sum} \mathcal{E}_i(t,x)\,.
$$
By construction, this velocity field is divergence-free and is regular $E^N \in \mathcal{C}([0,T] \times \mathbb{R}^3 )$, $E^N(t,\cdot ) \in  \mathcal{C}^1(\mathbb{R}^3) $ for all $0 \leq t \leq T $. Moreover is satisfies for all $t \in[0,T]$
$$E^N(t,x_i(t)) = E_i(t)  \text{ for all  $1 \leq i \leq N $,} $$ 
\begin{equation}\label{bound_E^N}
\| E^N(t,\cdot) \|_\infty \leq \,C_\chi\, \underset{i}{\max}\,|E_i(t)| \lesssim d_{\min}^N.
 \end{equation}
The only statement that needs further explanation is \eqref{bound_E^N}. 
For all $x\in B(x_i(t),R_i)$ we have
$$
\mathcal{E}_i(t,x) = E_i(t),
$$
and for all $x \in B(x_i,2R) \setminus B(x_i,R)$, direct computations yields 
\begin{multline*}
\mathcal{E}_i(t,x)= \frac{1}{2}\Big [2 \chi \left (\frac{x-x_i(t)}{R}\right )  \mathbb{I}_3- \frac{1}{R}\nabla  \chi \left (\frac{x-x_i(t)}{R}\right )  \otimes (x-x_i(t))\\
+  \frac{1}{R} (x-x_i(t))\cdot  \nabla  \chi \left (\frac{x-x_i(t)}{R}\right ) \mathbb{I}_3   \Big] E_i(t).
\end{multline*}
Therefore
$$
|\mathcal{E}_i(t,x)| \leq C \left [\|\chi \|_\infty +  \| \nabla \chi\|_\infty  \right ] | E_i(t)|.
$$
We can now state the following proposition.
\begin{proposition}\label{flow}
For arbitrary $N$ we have that $\kappa g + \mathcal{K}^N \rho^N + E^N \in \mathcal{C}([0,T] \times \mathbb{R}^3)$ and $\nabla \mathcal{K}^N \rho^N + \nabla E^N \in \mathcal{C}([0,T] \times \mathbb{R}^3) $. Moreover, the velocity field satisfies 
\begin{equation}\label{croissance_lineaire}
|\kappa g + \mathcal{K}^N \rho^N(t,x) + E^N (t,x)| \leq C  \:,\: \forall (t,x) \in [0,T]  \times \mathbb{R}^3,
\end{equation}
for some positive constant $C$ independent of $N$.
\end{proposition}
\begin{proof}
As the kernel is regularized, the two first properties are satisfied by construction. For all $ (t,x) \in [0,T] \times \mathbb{R}^3$ we have 
\begin{align*}
 \mathcal{K}^N\rho^N(x) &=  \frac{6\pi r_0 \kappa}{N} \underset{i}{\sum}\psi^N(x) \Phi(x-x_i(t))\,,\\
&=  \frac{6\pi r_0 \kappa}{N} \underset{i}{\sum} \psi^N(x) 1_{\{|x_i(t)-x| > \frac{d_{\min}^N(0)}{2}\}} \Phi(x-x_i(t)).
\end{align*}
We set $\mathcal{I}(t,x) = \{  1\leq i \leq N \,,\, |x_i(t)-x| > \frac{d_{\min}^N(0)}{2}\}$. Reproducing the arguments of Lemma \ref{JO} for $k=1$ together with assumptions \eqref{hyp1}, \eqref{compatibility} and Theorem \ref{thm1} yields
\begin{align*}
\left | \mathcal{K}^N\rho^N(x)\right | &\lesssim \frac{1}{N} \underset{\mathcal{I}(t,x)}{\sum} \frac{1}{|x-x_i(t)|}\,,\\
& \lesssim \bar{M} \frac{|\lambda^N|^3}{d_{\min}^N(0)} +\bar{M}^{1/3} \,,\\
&\lesssim \bar{M} \frac{|\lambda^N|^3}{|d_{\min}^N(0)|^2}d_{\min}^N(0) +\bar{M}^{1/3} \,,\\
&\lesssim 1\,.
\end{align*}
Furthermore, the velocity field $E^N$ is uniformly bounded according to \eqref{bound_E^N}.
\end{proof}
This allows us to state the following result.
\begin{theorem}\label{thm_VlasovN}
$\rho^N$ is a weak solution of 
\begin{equation}\label{Vlasov_StokesN}
\left\{
\begin{array}{rcl}
\frac{\partial \rho^N}{\partial t }+{\div}( ( \kappa g + \mathcal{K}^N\rho^N + E^N ) \rho^N ) &=& 0\,, \, \\
\rho^N(0,\cdot) &= & \rho^N_0 \,,
\end{array}
\right.
\end{equation}
on $[0,T]\times \mathbb{R}^3$.
Moreover, the characteristic flow defined for all $ s,t \in [0,T]$ by
\begin{equation}\label{flow_XN}
\left \{
\begin{array}{rcl}
\partial_s X^N(s,t,x) & = & \kappa g + \mathcal{K}^N\rho^N(s, X^N(s,t,x)) + E^N (s,X^N(s,t,x)), \\
X^N(t,t,x) & = & x,
\end{array}
\right.
\end{equation}
is of class $\mathcal{C}^1$ for all $N\geq 1$ and the following classical formula holds true:
\begin{equation}\label{formula_rho^N}
\rho^N(t,\cdot)= X^N(t,0,\cdot) \# \rho^N_0 \:\: \forall\, t \in[0,T].
\end{equation}
\end{theorem}
\begin{proof}
As $V(t,x):=\kappa g + \mathcal{K}^N\rho^N(t,x) + E^N(t,x)\in \mathcal{C}^1([0,T] \times \mathbb{R}^3) $ is defined such that $V(t,x_i(t))=V_i\,,\, \: \forall \, 1\leq i \neq N$ this ensures that for all test function $\phi \in \mathcal{C}^\infty_c ( [0,T] \times \mathbb{R}^3 )$:
$$
\int_0^T \int_{\mathbb{R}^3} \left ( \partial_t \phi(t,x) + \nabla \phi(t,x) \cdot \left [ \kappa g + \mathcal{K}^N\rho^N(t,x) + E^N(t,x) \right ] \right ) \rho^N(dx,t) dt  =0,
$$
thus, $\rho^N$ is a weak solution for \eqref{Vlasov_StokesN}.\\
According to Proposition \ref{flow}, the ode governing the characteristic flow satisfies the assumptions of the Cauchy-Lipschitz theorem. Therefore, the ode admits a unique maximal solution  $X^N \in \mathcal{C}^1([0,T] \times [0,T] \times \mathbb{R}^3)$ thanks to formula \eqref{croissance_lineaire}.
Equality \eqref{formula_rho^N} holds true thanks to the classical theory for transport equations.
\end{proof}
\section{Proof of Theorem \ref{thm2}}
At this point, we proved that the particles interact two by two with an interaction force given by the Oseen-tensor with an additional error term.
\begin{equation}
\left \{
\begin{array}{rcl}
\dot{x}_i(t) &= & V_i(t),\\
V_i(t) & = & \kappa g + 6\pi R \underset{i\neq j}{\sum} \Phi(x_i(t)-x_j(t))+E^N(t,x_i(t)).
\end{array}
\right.
\end{equation}
We want to estimate the Wasserstein distance $W_1(\rho^N(t,\cdot),\rho(t,\cdot))$ for all time $0\leq t \leq T$. To this end, we follow the ideas of \cite{Hauray} and \cite{HJ} and show that the additional error term $E^N$ can be controled. As in \cite{HJ}, we introduce an intermediate density $\bar{\rho}^N$.  
\subsection{Step 1. Estimate of the distance between $\rho$ and $\bar{\rho}^N$}
We define $\bar{\rho}^N_0$ as the regularized density  of $\rho^N_0$: 
$$\bar{\rho}^N_0:= \rho^N_0 * \chi_{\lambda^N} $$ 
where $\chi_{\lambda^N}(x):= \frac{1}{|\lambda^N|^3} \chi\left (\frac{x}{\lambda^N} \right) $ a mollifier compactly supported. Note that the support of $\chi$ is not important, we consider for instance $\chi$ such that $\supp \chi=B(0,1)$.
We emphasize that the regularized density is uniformly bounded
\begin{align}\label{bound_bar_rhoNinfty}
\bar{\rho}^N_0(x)& = \int \frac{1}{|\lambda^N|^3} \chi \left (  \frac{x-y}{\lambda^N} \right) \rho_0^N(dy)\,, \notag\\
&= \frac{1}{N|\lambda^N|^3}\underset{i=1}{\overset{N}{\sum}} \chi\left (  \frac{x-x_i(0)}{\lambda^N} \right)\,,\notag \\
& \leq \frac{1}{N|\lambda^N|^3} \| \chi\|_\infty \sup_x \# \{i\in \{1,\dots,N\}\,,\, x_i(0) \in B(x,\lambda^N) \}\,,\notag\\
& \leq \|\chi\|_\infty  \bar{M},
\end{align}
according to assumption \eqref{bound_concentration}. Moreover, we have 
\begin{align}\label{bound_bar_rhoN1}
\int_{\mathbb{R}^3}\bar{\rho}^N_0(x) dx& =\frac{1}{N |\lambda^N|^3} \underset{i=1}{\overset{N}{\sum}}\int_{B(x_i(0),\lambda^N)} \chi\left (  \frac{x-x_i(0)}{\lambda^N} \right)dx\,,\notag \\
&=1.
\end{align}
Now, we define $\bar{\rho}^N$ as the unique solution to  equation \eqref{Vlasov_Stokes} associated to the initial data $\bar{\rho}^N_0$. The stability Theorem \ref{stability_Loeper} allows us to compare $\rho$ and $\bar{\rho}^N$: 
$$
W_1(\rho(t,\cdot), \bar{\rho}^N(t,\cdot)) \leq W_1(\rho_0, \bar{\rho}_0^N)e^{C t} ,
$$
where $C=C(\|\chi\|_\infty, \bar{M},\|\rho_0\|_{\infty})$.
We split the distance $W_1(\rho_0, \bar{\rho}_0^N)$ as follows 
$$
W_1(\rho_0, \bar{\rho}_0^N) \leq W_1(\rho_0,{\rho}_0^N)+ W_1(\rho_0^N,\bar{\rho}_0^N),
$$
and use the fact that
$$
W_1(\rho_0^N,\bar{\rho}_0^N) \leq W_\infty(\rho_0^N,\bar{\rho}_0^N),
$$
together with \cite[Proposition 1]{HJ}
\begin{equation}\label{born_wasserstein0_bis}
W_\infty(\rho_0^N,\bar{\rho}_0^N)\leq {C}\lambda^N,
\end{equation}
to get 
\begin{equation}\label{estimate_1}
W_1(\rho(t,\cdot), \bar{\rho}^N(t,\cdot)) \lesssim \left ( \lambda^N+W_1(\rho_0,{\rho}_0^N) \right) e^{C t}.
\end{equation}
\subsection{Step 2. Estimate of the distance between $\bar{\rho}^N$ to $\rho^N$}
It remains to estimate $W_1(\rho^N(t,\cdot), \bar{\rho}^N(t,\cdot))$. We have the following result. 
\begin{lemma}\label{lemme_Wasserstein}
For $N$ large enough, there exists a positive constant $C$ such that for all $t\in[0,T]$
$$
W_1(\rho^N(t,\cdot), \bar{\rho}^N(t,\cdot)) \lesssim \left( \lambda^N  + t  d_{\min}^N \right)e^{Ct}.$$
\end{lemma}
Theorem \ref{thm2} is a consequence of estimate \eqref{estimate_1} and Lemma \ref{lemme_Wasserstein}. The rest of this Section is devoted to proving the above lemma.
\begin{proof}[Proof of Lemma \ref{lemme_Wasserstein}]
According to Theorems \ref{thm_Vlasov} and \ref{thm_VlasovN} we have the explicit formulas for all $s\,,\,t \in [0,T]$ 
\begin{eqnarray*}
 \bar{\rho}^N(t,\cdot)& =  &X(t,s,\cdot ) \# \bar{\rho}^N_s, \\
 \rho^N(t,\cdot) & = & X^N(t,s,\cdot ) \# \rho_s^N.
 \end{eqnarray*}
At $t=0$ we have the existence of an optimal transport map $T_0$ from $\bar{\rho}_0^N$ to $\rho_0^N$ thanks to Theorem \ref{existence_transport_map}
$$\rho_0^N= T_0 \#  \bar{\rho}^N_0,$$
satisfying
$$
W_\infty (\bar{\rho}^N_0,\rho^N_0) = \bar{\rho}^N_0 -  \esssup |T_0(x)-x|.
$$
We construct then a transport map $T_t$ from $\bar{\rho}^N$ to $\rho^N$ at all time $t\in[0,T]$ by following $T_0$ along the two flows $X$ and $X^N$
$$
T_t= X^N(t,0,\cdot) \circ T_0 \circ X(0,t,\cdot).
$$
One can remark that for all $0\leq s\leq t$:
\begin{eqnarray*}
 T_t & = & X^N(t,s,\cdot) \circ T_s \circ X(s,t,\cdot)\\
 \rho^N(t,\cdot)& = & T_t \# \bar{\rho}^N(t,\cdot).
\end{eqnarray*}
As in \cite{HJ} we set then 
$$f^N(t):=\underset{s\leq t}{\sup} \ \bar{\rho}^N(t,\cdot) -  \esssup |T_s(x)-x|,
$$
so that 
$$
W_\infty ( \rho^N(t,\cdot), \bar{\rho}^N(t,\cdot)) \leq f^N(t),
$$and thanks to \eqref{born_wasserstein0_bis} we have
\begin{equation}\label{born_wasserstein0}
f^N(0)= W_\infty (\bar{\rho}^N_0,\rho^N_0)\leq C \lambda^N.
\end{equation}
We reproduce the same steps as in \cite{HJ} and introduce the following notation for a generic ``particle'' of the continuous system with position $x_t$ at time $t$ such that
$$x_s=X(s,t,x_t),$$
we fix in what follows $0 \leq t_2 \leq t_1 $ and recall the following formula 
$$
T_{t_1}\circ X(t_1,t_2,\cdot)= X^N(t_1,t_2,\cdot) \circ T_{t_2}.
$$
We aim now to estimate $|T_{t_1}(x_{t_1})-x_{t_1}|$ for all test particle $x_{t_1}$ 
\begin{align*}
T_{t_1}(x_{t_1}) - x_{t_1} & = X^N(t_1,t_2,T_{t_2}(x_{t_2}))- X(t_1,t_2,x_{t_2}), \\
& = T_{t_2}(x_{t_2})- x_{t_2} + \int_{t_2}^{t_1} \dot{X}^N(s,t_2,T_{t_2}(x_{t_2})) - \dot{X}(s,t_2,x_{t_2}) ds, \\
& =  T_{t_2}(x_{t_2})- x_{t_2} + \int_{t_2}^{t_1} \big ( [ \mathcal{K}^N \rho^N +E^N](s,X^N(s,t_2,T_{t_2}(x_{t_2}))),\\
& - \mathcal{K}\bar{\rho}^N(s,x_s) ) \big ) ds, \\
& =  T_{t_2}(x_{t_2})- x_{t_2} +  \int_{t_2}^{t_1} \big(  [\mathcal{K}^N \rho^N+E^N](s,T_s(x_s)) - \mathcal{K}\bar{\rho}^N(s,x_s )) \big ) ds, \\
& =  T_{t_2}(x_{t_2})- x_{t_2} + \int_{t_2}^{t_1} E^N(s,T_s(x_s)) ds, \\
& +\int_{t_2}^{t_1} \int_{\mathbb{R}^3} 6 \pi r_0 \kappa \left (   \psi^N\Phi (T_s(x_s)- T_s(y)  ) - \Phi(x_s-y) \right )g \bar{\rho}^N(s,dy) ds,
\end{align*}
where we used the fact that $\rho^N_s = T_s \#\bar{\rho}^N_s$ to get
\begin{align*}
\mathcal{K}^N\rho^N(s,T_s(x_s)) & = 6\pi r_0 \kappa \int_{\mathbb{R}^3} \psi^N\Phi(T_s(x_s)-y) g \rho^N(s,dy)\,, \\
& =6\pi r_0 \kappa \int_{\mathbb{R}^3} \psi^N\Phi(T_s(x_s) -T_s(y))g\bar{\rho}^N(s,dy).
\end{align*}
We set then $t_1=t$ and $t_2=t_1-\tau = t- \tau $ , $\tau>0$. We obtain for almost every $x_{t}$
\begin{align*}
|T_{t}(x_{t})- x_{t}| &\leq |T_{t-\tau}(x_{t-\tau})- x_{t-\tau}| + \tau \|E^N(t)\|_\infty,\\
&  +6\pi r_0 \kappa |g| \int_{t-\tau}^{t} \int_{\mathbb{R}^3} \left |   \psi^N\Phi (T_s(x_s)- T_s(y)  ) - \Phi(x_s-y) \right | \bar{\rho}^N(s,dy) ds, \\
& \leq f^N(t-\tau) + \tau \|E^N(t)\|_\infty\,, \\
&+C\, \int_{t-\tau}^{t} \int_{\mathbb{R}^3} \left |   \psi^N\Phi (T_s(x_s)- T_s(y_s)  ) - \Phi(x_s-y_s) \right | \bar{\rho}^N(t,d y_t) ds,
\end{align*}
here we used Remark \ref{rem2} with $y_s=X(s,t,y_t)$. In addition we defined
$$\|E^N(t)\|_\infty:= \underset{0\leq s \leq t }{\sup} \|E^N(s,\cdot)\|_\infty.$$
This being true for almost every $x_{t}$ we obtain
\begin{multline}
f^N(t)\leq f^N(t-\tau) + \tau \|E^N(t)\|_\infty \\ +C \,\underset{x_{t}}{ \esssup}\int_{t-\tau }^{t} \int_{\mathbb{R}^3} \left |   \psi^N\Phi (T_s(x_s)- T_s(y_s)  ) - \Phi(x_s-y_s) \right | \bar{\rho}^N(t,d y_t) ds. 
\end{multline}
Hence, it remains to control the last quantity. We split the integral on $\mathbb{R}^3$ into two terms: the first one denoted $J_1$ is the integral over the subset $I$ and the second one denoted $J_2$ the integral over $\mathbb{R}^3 \setminus I$ where
$$
I= \{y_t \::\: |x_t-y_t|\geq 4 f^N(t) e^{\tau L} \},
$$
where $L$ will be defined later.
\paragraph{\textbf{Step 1: Estimate of $J_1$}}
\text { } \\
For all $t-\tau \leq s \leq t$,  we have
\begin{align*}
|x_s-y_s|& \geq |x_t-y_t| - \int_{s}^t |\dot{X}(t',t,x_t)- \dot{X}(t',t,y_t)| dt', \\
& \geq  |x_t-y_t| - \int_s^t | \mathcal{K}\bar{\rho}^N (t', X(t',t,x_t)) -  \mathcal{K}\bar{\rho}^N (t', X(t',t,y_t))| dt', \\
& \geq |x_t-y_t| - \text{Lip }( \mathcal{K} \bar{\rho}^N)  \int_s^t | X(t',t,x_t) -  X(t',t,y_t)| dt'. 
\end{align*}
Using Remarks \ref{rem2} and \ref{remarque2}, formula \eqref{formula_rho} and the uniform bounds \eqref{bound_bar_rhoNinfty}, \eqref{bound_bar_rhoN1}, the Lipschitz constant of $\mathcal{K} \bar{\rho}^N$ is uniformly bounded. This allows us to define the constant $L$ as 
$$
\text{Lip }( \mathcal{K} \bar{\rho}^N) \leq C \|\bar{\rho}_0^N\|_{L^\infty( L^\infty \cap L^1)} \leq L\,.
$$
Applying Gronwall's inequality yields for all $0 \leq t-\tau \leq s \leq t$
$$
|x_s-y_s| \geq |x_t-y_t| e^{-L(t-s)}.
$$
We can make precise now the constant $L := \text{Lip }( \mathcal{K} \bar{\rho}^N)$ which is uniformly bounded with respect to $N$ and $t\in[0,T]$. \\
We have for all $0 \leq t-\tau \leq s \leq t$ and $\tau$ small enough
\begin{equation}\label{formule__1}
|x_s-y_s| \geq |x_t-y_t| e^{-L (t-s)} \geq |x_t-y_t| e^{- L \tau} \geq \frac{1}{2} |x_t-y_t|.
\end{equation} 
Analogously, for almost all $x_s$ and $y_s$ 
\begin{multline*}
|T_s(x_s)-T_s(y_s)| \geq |x_s-y_s| - |T_s(x_s)-x_s| - |T_s(y_s)-y_s|\,,\\ \geq |x_s-y_s| - 2 f^N(s) \geq |x_s-y_s| - 2 f^N(t),
\end{multline*}
where we used the fact that $f^N(t)\geq f^N(s)$. According to the definition of $I= \{y_t \::\: |x_t-y_t|\geq 4 f^N(t) e^{\tau L} \}$, this yields for $\tau$ small enough
 \begin{equation}\label{formule__2}
|T_s(x_s)-T_s(y_s)| \geq \frac{1}{4}|x_t-y_t|.
\end{equation}
Moreover, recall that $T_s(x_s)$ and $T_s(y_s)$ are in the support of $\rho^N(s,\cdot)$ \textit{i.e.} there exists $i\,,\, j$ such that $T_s(x_s)=x_i(s)$ and $T_s(y_s)=x_j(s)$. In addition, estimate \eqref{formule__2} and the definition of $I$ ensures that $i\neq j$. We have then 
\begin{equation}\label{formule__3}
\psi^N\Phi(T_s(x_s)-T_s(y_s)) = \Phi(T_s(x_s)-T_s(y_s)).
\end{equation}
Finally, using estimates \eqref{formule__1}, \eqref{formule__2}, formula \eqref{formule__3} and the Lipschitz-like estimate \eqref{lipschitz} for $\Phi$ we obtain
\begin{align*}
J_1 & = \int_{I} \int_{t-\tau }^{t}  \left |   \Phi (T_s(x_s)- T_s(y_s)  ) - \Phi(x_s-y_s) \right | ds \bar{\rho}^N(t,dy_t), \\
& \leq C \int_{I} \int_{t-\tau }^{t}  \frac{|x_s-T_s(x)|+ | y_s-T_s(y)|}{\min (|x_s-y_s|^2\,,\, |T_s(x)-T_s(y)|^2)} ds \bar{\rho}^N(t,dy_t), \\
& \leq C f^N(t)  \tau \int_I \frac{1}{|x_t-y_t|^2}\bar{\rho}^N(t,dy_t),\\
& \leq C \tau f^N(t) \|\bar{\rho}^N(t)\|_{L^\infty\cap L^1}\,,\\
& \leq C \tau f^N(t)\|\bar{\rho}^N_0\|_{L^\infty\cap L^1}\,,\\
& \leq C \tau f^N(t)\,,
\end{align*} 
where we used Remark \ref{rem2}, formula \eqref{formula_rho} and the uniform bounds \eqref{bound_bar_rhoNinfty}, \eqref{bound_bar_rhoN1}.
\paragraph{\textbf{Step 2: Estimate of $J_2$}}
 \text { } \\
We focus now on 
$$
J_2:= \underset{x_{t}}{\esssup}\int_{t-\tau }^{t} \int_{{}^c I} \left |   \psi^N\Phi (T_s(x_s)- T_s(y_s)  ) - \Phi(x_s-y_s) \right | \bar{\rho}^N(t,d y_t) ds.
$$
Again $T_s(x_s)$ and $T_s(y_s)$ are in the support of $\rho^N(s,\cdot)$ \textit{i.e.} there exists $i\,,\, j$ such that $T_s(x_s)=x_i(s)$ and $T_s(y_s)=x_j(s)$. Moreover if $i=j$ then $\psi^N \Phi(T_s(x_s)-T_s(y_s))=0$. Hence in all cases we have
\begin{align*}
\left|\Phi(x_s-y_s) - \psi^N \Phi(T_s(x_s)-T_s(y_s)) \right |&  \leq |\Phi(x_s-y_s)|+|\psi^N \Phi(T_s(x_s)-T_s(y_s))|\,,\\
& \leq C \left ( \frac{1}{|x_s-y_s|} + \frac{1}{d_{\min}^N(s)} \right ),
\end{align*}
applying the change of variable $y_t=X(t,s,y_s)$ we get
\begin{align*}
\int_{{}^c I} \int_{t-\tau}^{t} \frac{1}{|x_s-y_s|} ds \bar{\rho}^N(t,dy_t) & \leq \|\bar{\rho}^N\|_\infty
 \int_{t-\tau}^{t}\int_{{}^c I} \frac{1}{|x_s-y_s|}  dy_t ds, \\
& = C \int_{t-\tau}^{t} \int_{X(t,s,{}^c I)}  \frac{1}{|x_s-y_s|}  dy_s ds. 
\end{align*}
Denote $K = X(t,s,{}^c I)$, as the flow $X$ preserves the Lebesgue measure we have $|K| = |{}^c I |$. For all $s \in [t-\tau,t]$ and  $a>0$ a direct computation yields
\begin{align*}
\int_{K}  \frac{1}{|x_s-y_s|}  dy_s &=\left ( \int_{K \cap B(x_s,a)} + \int_{K \cap {}^cB(x,a)} \right )\frac{1}{|x_s-y_s|}  dy_s ,\\
& \leq C a^2 + \frac{1}{a} |K| ,
\end{align*}
we choose then $a^3=|K|= |{}^c I|\leq C \left|f^N(t)\right|^3 e^{3L\tau}$ to get 
\begin{equation}\label{terme_1}
\int_{{}^c I} \int_{t-\tau}^{t} \frac{1}{|x_s-y_s|} ds \bar{\rho}^N(t,dy_t) \leq C \tau \left|f^N(t)\right|^2 e^{2 L \tau }.
\end{equation}
For the remaining term we apply Theorem \ref{thm1} and get for all $t - \tau \leq s \leq t $ 
\begin{align*}
\int_{{}^c I} \int_{t-\tau}^{t} \frac{1}{d_{\min}^N(s)} ds \bar{\rho}^N(t,dy_t) &\leq \frac{2}{d_{\min}^N(0)} \int_{{}^c I} \int_{t-\tau}^{t}  ds \bar{\rho}^N(t,dy_t), \\
& \leq C \tau \frac{2 e^{3\tau L}}{d_{\min}^N(0)}  \left|f^N(t)\right|^3.
\end{align*} 
 \paragraph{\textbf{Conclusion}}
 \text { } \\
Gathering these bounds, there exists a constant $K>0$ independent of $N$ such that for $\tau$ small enough and $ 0 < t \leq T $ 
$$
f^N(t) \leq f^N(t-\tau)+ \tau \|E^N (t)\|_\infty + K \tau f^N(t)\left [1 + f^N(t) +\frac{\left|f^N(t)\right|^2}{d_{\min}^N(0)}   \right  ].
$$
We can now apply a discrete Gronwall argument: Note that at time $t=0$, assumption \eqref{hyp5} and formula \eqref{born_wasserstein0} ensures the existence of a positive constant $C_1>1$ such that 
$$
1+ f^N(0)+\frac{\left|f^N(0)\right|^2}{d_{\min}^N(0)} \leq \frac{C_1}{K},
$$ 
hence, we define $T^*\leq T$ as the maximal time for which 
\begin{equation}\label{blow_up_formula}
 1+ f^N(t)+\frac{\left|f^N(t)\right|^2}{d_{\min}^N(0)} \leq \frac{C_1}{K} \,\:\: \forall t \in[0,T^*[.
\end{equation}
Note that $T^*$ a priori depends on $N$, the purpose is to show that this is not the case.
We obtain for all $t\in[0,T^*[$
$$
f^N(t) \leq f^N(t-\tau) + C_1 \tau f^N(t) + \tau \|E^N\|_\infty.
$$
If $\tau$ is small enough we can write 
$$
f^N(t) \leq (1-C_1 \tau)^{-1} f^N(t-\tau) + \frac{\tau}{1-C_1 \tau } \|E^N\|_\infty,
$$
iterating the formula we obtain for $M \in \mathbb{N}^*$
\begin{align*}
f^N(t)& \leq (1-C_1 \tau )^{-M} f^N(t-M\tau) + \tau \underset{k=1}{\overset{M}{\sum}} \frac{1}{(1-C_1 \tau )^k} \|E^N\|_\infty,\\
& \leq (1-C_1 \tau )^{-M} f^N(t-M\tau) + \tau \underset{k=1}{\overset{M}{\sum}} e^{2C_1 \tau k} \|E^N\|_\infty.
\end{align*}
Thanks to the bound $ \frac{1}{1-C_1 \tau} \leq e^{2C_1 \tau}$ for $\tau$ small enough.
We set then $t- M \tau =0$ to get
$$
f^N(t) \leq (1-C_1 \frac{t}{M})^{-M} f^N(0) + \frac{t}{M}  
\overset{M}{\underset{k=1}{\sum}}  e^{2C_1 \frac{t}{M} k} \|E^N\|_\infty.
$$
As $ e^{2C_1 \frac{t}{M} k}  \leq e^{2C_1 t } $ for all $1\leq k \leq M$  the second term yields 
\begin{align*}
\frac{t}{M}  
\overset{M}{\underset{k=1}{\sum}}  e^{2C_1 \frac{t}{M} k} \|E^N\| &\leq t e^{2C_1 t } \|E^N\|_\infty,
\end{align*}
and for $M$ sufficiently large
$$
(1-C_1 \frac{t}{M} )^{-M}  \leq e^{2C_1t}.
$$
Finally for all $t\in[0,T^*[$
$$
f^N(t) \leq f^N(0) e^{2C_1t} +  t e^{2C_1t} \|E^N\|_\infty.$$
In particular we have for all $t \in [0,T^*[$
\begin{equation*}
\begin{split}
&f^N(t)+\frac{\left|f^N(t)\right|^2}{d_{\min}^N(0)}\\
&\leq  f^N(0)e^{2C_1t}+\|E^N\|_\infty Te^{2C_1t}+2\frac{ \left|f^N(0)\right|^2e^{4C_1t}+  \|E^N\|^2_\infty T^2e^{4C_1t} }{d_{\min}^N(0)}\,,\\
& \leq e^{4C_1T}(2+T+2T^2)\left (f^N(0)+\|E^N\|_\infty+ \frac{\left|f^N(0)\right|^2+ \|E^N\|^2_\infty }{d_{\min}^N(0)}\right).
\end{split}
\end{equation*}
Since we have $f^N(0) = O\left (\lambda^N \right)$ and thanks to \eqref{bound_E^N}
\begin{align*}
\frac{\left|f^N(0)\right|^2+ \|E^N\|^2_\infty }{d_{\min}^N(0)}& \lesssim \frac{|\lambda^N|^2}{d_{\min}^N(0)} + d_{\min}^N \,,
\end{align*}
which vanishes according to assumption \eqref{compatibility} and \eqref{hyp5}. This shows that we can take $N$ large enough and depending on $T$, $K$ and $C_1$ such that $T^* \to T$ and formula \eqref{blow_up_formula} holds true up to time $T$. Hence, for $N$ large enough we have for all $t\in[0,T]$
$$
f^N(t) \leq f^N(0) e^{2C_1t} + t e^{2C_1t} \|E^N\|_\infty.$$
Using \eqref{born_wasserstein0} and the fact that $W_1(\rho^N,\bar{\rho}^N) \leq W_\infty(\rho^N,\bar{\rho}^N)\leq f^N$, this implies Lemma \ref{lemme_Wasserstein}.
\end{proof} 
\appendix
\section{Technical lemmas}
We state here an important lemma which is the extension of \cite[Lemma 2.1]{JO} to the new assumptions on the dilution regime introduced in \cite{Hillairet}. We introduce $\tilde{\rho}^N$ an approximation of $\rho^N$ defined as
\begin{equation}\label{tilderhoN}
\tilde{\rho}^N(t,x):= \frac{1}{N} \underset{i=1}{\overset{N}{\sum}} \frac{1_{B_\infty(x_i,\lambda^N/3)}}{\left|B_\infty(x_i,\lambda^N/3) \right|}\,.
\end{equation}
$\tilde{\rho}^N$ is $L^\infty$ and using \eqref{bound_concentration}, one can check that 
\begin{multline}\label{bound_tildethoN}
\|\tilde{\rho}^N\|_{L^\infty} \lesssim \frac{1}{N |\lambda^N|^3}\underset{x \in \mathbb{R}^3}{\sup} \# \left\{ i \in \{ 1 , \cdots, N\} \text{ such that } x_i \in B_\infty(x,\lambda^N/3) \right \}\\ \lesssim \frac{M^N}{N |\lambda^N|^3}  \lesssim \bar{M}\,.
\end{multline}
Moreover, $\tilde{\rho}^N$ is $L^1$ and we have $\|\tilde{\rho}^N\|_{L^1}=1$ by construction.
\begin{lemma}\label{JO}
For all $k \in[0,2]$, under assumptions \eqref{bound_concentration}, \eqref{compatibility},
 if $N$ is large enough, there exists a positive constant $C>0$ such that for all fixed $1 \leq i \leq N$:
\begin{equation}
\frac{1}{N} \underset{j\neq i}{\sum} \frac{1}{d_{ij}^k} \leq C \bar{M}\frac{|\lambda^N|^3}{|d_{\min}^N|^k} +\bar{M}^{k/3}.
\end{equation}
Moreover, if $k=3$ we have
$$\frac{1}{N} \underset{j\neq i}{\sum} \frac{1}{d_{ij}^3} \leq C  \bar{M}\left(\frac{|\lambda^N|^3}{|d_{\min}^N|^3}+| \log (\bar{M}^{1/3} \lambda^N)|+1 \right).
$$
\end{lemma}
\begin{proof}
We fix $i =1$ and the same holds true for all $1 \leq i \leq N$.
We use the following shortcut 
$$\mathcal{I}_1:=\{j\in \{1,\cdots,N \} \text{ such that } |x_1-x_j|_\infty \leq \lambda^N \}.$$
The sum can be written as follows:
\begin{align*}
\frac{1}{N} \underset{j\neq 1}{\sum} \frac{1}{d_{1j}^k}&= \frac{1}{N} \underset{\underset{j\neq 1}{j \in \mathcal{I}_1}}{\sum} \frac{1}{d_{1j}^k} + \frac{1}{N} \underset{j\not \in \mathcal{I}_1}{\sum} \frac{1}{d_{1j}^k} \,,\\
& \leq \frac{1}{N} \frac{M^N}{|d_{\min}^N|^k}+\frac{1}{N} \underset{j\not \in \mathcal{I}_1}{\sum} \frac{1}{d_{1j}^k} \,,\\
&\leq \bar{M} \frac{|\lambda^N|^3}{|d_{\min}^N|^k}+\frac{1}{N} \underset{j\not \in \mathcal{I}_1}{\sum} \frac{1}{d_{1j}^k} \,.
\end{align*}
For the second term in the right hand side, note that, for all $ y \in B_\infty(x_j,\lambda^N/3)$, $ j\not \in \mathcal{I}_1 $ we have 
$$
|x_1-y|_\infty \geq |x_1-x_j|_\infty - |x_j -y|_\infty \geq 2/3 \lambda^N\,, 
$$
this yields
$$
|x_1-x_j|_\infty \geq |x_1-y|_\infty - \lambda^N/3 \geq |x_1-y|_\infty/2 \,.
$$
Hence, we have for all constant $L>2/3\lambda^N$
\begin{align*}
\frac{1}{N} \underset{j\not \in \mathcal{I}_1}{\sum} \frac{1}{d_{1j}^k} &\leq \frac{2^k}{N} \underset{j\not \in \mathcal{I}_1}{\sum} \int_{B_\infty(x_j, \lambda^N/3)} \frac{1}{|B_\infty(x_j,\lambda^N/3)|} \frac{1}{|x_1-y|^k} dy \,,\\ 
&\lesssim\int_{{}^c B(x_1,2/3 \lambda^N)}\frac{1}{|x_1-y|^k}\tilde{\rho}^N(t,dy)\,,\\
& \leq \|\tilde{\rho}^N\|_{L^\infty} \int^L_{2/3|\lambda^N|} r^{2-k} dr +  \int_{{}^c B(x_1,L)} \frac{1}{|x_1-y|^k}\tilde{\rho}^N(t,dy) \,,\\
 &\leq \|\tilde{\rho}^N\|_{L^\infty} \frac{L^{3-k}-\left( 2/3|\lambda^N|\right)^{3-k}}{3-k}+  \frac{\|\tilde{\rho}^N\|_{L^1}}{L^k} \,,\\
 &\lesssim \bar{M} \frac{L^{3-k}}{3-k}+  \frac{1}{L^k} \,.
\end{align*}
One can show that the optimal constant $L>2/3\lambda^N$ is $L= \frac{1}{\bar{M}^{1/3}}$. Since $\underset{N \to \infty}{\lim}\lambda^N=0$, this choice of $L$ is possible for $N$ large enough such that $ \lambda^N<\frac{3}{2\bar{M}^{1/3}}$. Hence, we obtain
$$
\frac{1}{N} \underset{j\not \in \mathcal{I}_1}{\sum} \frac{1}{d_{1j}^k} \lesssim \frac{4-k}{3-k} \bar{M}^{k/3}.
$$
If $k=3$, we integrate the term $r^{-1}$ keeping the same value for $L$ as before
\begin{align*}
\frac{1}{N} \underset{j\not \in \mathcal{I}_1}{\sum} \frac{1}{d_{1j}^3}& \leq \|\tilde{\rho}^N\|_{L^\infty} \int^{\frac{1}{\bar{M}^{1/3}}}_{2/3|\lambda^N|} \frac{dr}{r} +  \int_{{}^c B(x_1,\frac{1}{\bar{M}^{1/3}})} \frac{1}{|x_1-y|^3}\tilde{\rho}^N(t,dy) \,,\\
 &\leq \bar{M} \left (\log \left(\frac{1}{\bar{M}^{1/3}|\lambda^N|} \right)+ \log\left(3/2 \right)\right)+  \bar{M} \,,\\
 &\leq 2 \bar{M}(| \log (\bar{M}^{1/3} \lambda^N)|+1 ) \,,
\end{align*}
for $N$ large enough to ensure $\frac{3}{2} \leq \frac{1}{\bar{M}^{1/3} |\lambda^N|}$.
\end{proof}
The following results are used for the control of the particle concentration $M^N$: 
$$
M^N(t) := \underset{x\in \mathbb{R}^3}{\sup} \Big\{ \# \big\{ i \in \{1,\cdots,N\} \text{ such that } x_i(t) \in \overline{B_\infty(x,\lambda^N)}  \big \}\Big\}.
$$
We recall the definition of $L^N$ introduced in \eqref{concentration_bis}: 
$$
L^N(t):= \underset{i}{\max} \# \left \{j \in \{1,\dots,N\} \text{ such that } |x_i(t)-x_j(t)|_\infty \leq \lambda^N \right \}.
$$
The following lemma shows that the two definitions are equivalent.
\begin{lemma}\label{equivalence}
We have 
$$
L^N(t) \leq M^N(t) \leq 8 L^N(t).
$$
\end{lemma}
\begin{proof}
The first inequality is trivial. To prove the second one note that we have: 
\begin{multline*}
\underset{x\in \mathbb{R}^3}{\sup} \Big\{ \# \big\{ i \in \{1,\cdots,N\} \text{ such that } x_i \in \overline{B_\infty(x,\lambda^N)}  \big \}\Big\} \leq \\
 8\, \underset{x\in \mathbb{R}^3}{\sup} \Big\{ \# \big\{ i \in \{1,\cdots,N\} \text{ such that } x_i \in \overline{B_\infty(x,{\lambda^N}/{2})}  \big \}\Big\}.
\end{multline*}
Indeed, for all $x \in \mathbb{R}^3 $ there exists $\bar{x}_k$, $k=1,\cdots,8$ such that 
$$
\overline{B_\infty\left (x,\lambda^N \right )}\subset \underset{k}{\overset{8}{\bigcup}} \overline{B_\infty \left (\bar{x}_k,\frac{\lambda^N}{2} \right )},
$$
this yields
\begin{multline*}
\big\{ i \in \{1,\cdots,N\} \text{ such that } x_i \in \overline{B_\infty(x, \lambda^N )}  \big \}\\
 \subset \underset{k}{\overset{8}{\bigcup}} \big\{ i \in \{1,\cdots,N\} \text{ such that } x_i \in \overline{B_\infty(\bar{x}_k,\lambda^N/2)}  \big \}.
\end{multline*}
Taking the supremum in the right hand side and then in the left one we obtain
\begin{multline}\label{inegalite1}
\underset{x\in \mathbb{R}^3}{\sup} \Big\{ \# \big\{ i \in \{1,\cdots,N\} \text{ such that } x_i \in \overline{B_\infty(x,\lambda^N)}  \big \}\Big\} \leq \\ 8\, \underset{x\in \mathbb{R}^3}{\sup} \Big\{ \# \big\{ i \in \{1,\cdots,N\} \text{ such that } x_i\in \overline{B_\infty(x,\lambda^N/2)}  \big \}\Big\}.
\end{multline}
Moreover, we remark that the supremum in the right hand side over all $x\in \mathbb{R}^3$ can be reduced to the supremum over $\underset{i}{\bigcup} \overline{B_\infty(x_i, \frac{\lambda^N}{2})}$. Now consider $x \in \underset{i}{\bigcup} \overline{B_\infty(x_i, \frac{\lambda^N}{2})}$, there exists $1 \leq i_0 \leq N$ such that $ |x-x_{i_0}|_\infty \leq \frac{\lambda^N}{2}$, we have then for all $j \neq i_0$ such that $ |x-x_j|_\infty \leq \frac{\lambda^N}{2} $: 
$$
|x_j-x_{i_0}|_\infty \leq |x_j-x|_\infty + |x-x_{i_0}|_\infty \leq \lambda^N,
$$
which means that for all $x  \in \underset{i}{\bigcup} \overline{B_\infty(x_i, \frac{\lambda^N}{2})}$ there exists $1 \leq i_0 \leq N$ such that 
\begin{multline*}
\big\{ 1 \leq j \leq N, \text{ such that } x_j \in\overline{B_\infty(x,{\lambda^N}/{2})}  \big \}\\ \subset \big \{ 1 \leq j \leq N, \text{ such that }  |x_j-x_{i_0}|_\infty \leq \lambda^N \big \}.
\end{multline*}
Taking the maximum over all $i_0$ in the right hand side, and then the supremum over all $x  \in \underset{i}{\bigcup} \overline{B_\infty(x_i, \frac{\lambda^N}{2})}$ we obtain 
\begin{multline}\label{inegalite2}
\underset{x}{\sup} \Big\{ \# \big\{ i \in \{1,\cdots,N\} \text{ such that } x_i\in \overline{B_\infty(x,{\lambda^N}/{2})}  \big \}\Big\} \leq \\
\underset{i}{\max} \# \left \{j \in \{1,\dots,N\} \setminus\{i\} \text{ such that } |x_i-x_j|_\infty \leq \lambda^N \right \}.
\end{multline}
Gathering inequality \eqref{inegalite1} and \eqref{inegalite2} concludes the proof.
\end{proof}
More generally we define for all $\beta>0$: 
$$
L^N_\beta(t):= \underset{i}{\max} \# \left \{j \in \{1,\dots,N\}  \text{ such that } |x_i(t)-x_j(t)|_\infty \leq {\beta}\lambda^N \right \},
$$
and
$$
M^N_\beta(t) := \underset{x\in \mathbb{R}^3}{\sup} \Big\{ \# \big\{ i \in \{1,\cdots,N\} \text{ such that } x_i(t) \in \overline{B_\infty(x,{\beta}\lambda^N)}  \big \}\Big\},
$$
with the notation $$M^N_1(t):= M^N(t)\:\:,\:\:L^N_1(t):= L^N(t).$$ The previous results yields 
\begin{corollary}\label{scaling_concentration}
For all $\beta>0$ and all $\alpha>1$ we have
$$
L^N_{\alpha \beta} (t) \leq 8 \lceil \alpha \rceil^3 L_\beta^N(t),
$$
where $\lceil \cdot \rceil$ denotes the ceiling function.
\end{corollary}
\begin{proof}
For sake of clarity we set $\beta=1$ and the proof remains the same for all $\beta>0$.  The idea is to show an equivalent formula for $M^N$ and use Lemma \ref{equivalence}. Analogously to the proof of Lemma \ref{equivalence}, for all  $x \in \mathbb{R}^3 $ there exists $\bar{x}_k$, $k=1,\cdots,\lfloor \lambda \rfloor^3$ such that 
$$
\overline{B_\infty \left (x,\alpha \lambda^N \right )}\subset \underset{k=1}{\overset{\lceil \alpha \rceil^3}{\bigcup}}\: \overline{B_\infty \left (\bar{x}_k,\lambda^N \right )}.
$$
This yields, with the definition of $M_\lambda^N$: 
$$
M^N_\alpha \leq \lceil \alpha \rceil^3 M^N(t).
$$
Finally, we apply Lemma \ref{equivalence} to get 
$$
L_\alpha^N(t) \leq M^N_\alpha(t) \leq  \lceil \alpha \rceil^3 M^N(t) \leq 8  \lceil \alpha \rceil^3 L^N(t)\,,
$$
which completes the proof.
\end{proof}

\section*{Acknowledgments}
The author would like to thank Matthieu Hillairet for introducing the subject and sharing his experience for overcoming the difficulties during this research. The author is also thankful to the referee for all his important suggestions.

\end{document}